\documentclass[11pt]{article}
\usepackage[utf8]{inputenc}
\usepackage{amsmath}
\usepackage{enumitem}
\usepackage{amsfonts}
\usepackage{setspace}
\usepackage{amssymb}
\usepackage{cancel}  
\usepackage{graphicx}
\usepackage{mathrsfs}
\usepackage{esint}
\usepackage{upref,amsthm,amsxtra,exscale}
\usepackage{cite}
\usepackage{tikz-cd}
\usepackage[colorlinks=true,urlcolor=blue,
citecolor=red,linkcolor=blue,linktocpage,pdfpagelabels,
bookmarksnumbered,bookmarksopen]{hyperref}
\usepackage{cleveref}
\usepackage[cm]{fullpage}

\usepackage{subcaption}
\usepackage{caption}

\numberwithin{equation}{section}
\def\dis{ }

\newtheorem{theorem}{Theorem}[section]
\newtheorem{proposition}[theorem]{Proposition}
\newtheorem{lemma}[theorem]{Lemma}
\newtheorem{corollary}[theorem]{Corollary}
\newtheorem{example}[theorem]{Example}

\newtheorem{remark}{Remark}
\newtheorem{definition}{Definition}

\title{An Intrinsic Sobolev Trace Theory for Riemannian Vector Bundles with Applications to Quasilinear Elliptic Equations}

\author{Carlos Daniel Velázquez-Mendoza, María de los Ángeles Sandoval-Romero, Romulo Diaz Carlos}
\begin{document}
 \maketitle

\begin{abstract}
Let \(\mathbf{E}\to M\) be a finite-rank Riemannian vector bundle over a compact Riemannian manifold with nonempty smooth boundary, equipped with a compatible connection. We develop an intrinsic trace theory for Sobolev sections defined through weak covariant derivatives. For every integer \(m\ge1\) and every \(1\le p<\infty\), we construct continuous trace operators for the covariant derivatives up to order \(m-1\) and characterize the Sobolev space with homogeneous Dirichlet boundary conditions as the kernel of the corresponding trace operator. As an application, we establish the existence of ground state solutions for a class of quasilinear elliptic equations on Riemannian vector bundles with generalized \((p,q)\)-growth. The proposed framework unifies several important geometric operators, including the Bochner \(p\)-Laplacian, the \((p,q)\)-Bochner Laplacian, and prescribed mean curvature-type operators.
\medskip


\medskip


\end{abstract}		

    \section{Introduction}
Sobolev spaces of sections of vector bundles form a natural analytic
framework for geometric partial differential equations. They arise in
the study of connection Laplacians, Hodge and Dirac operators, covariant
Schrödinger operators, and nonlinear systems whose unknowns take values
in a nontrivial bundle. The scalar theory on Riemannian manifolds is
classical; see, among many references, \cite{Aubin,Hebey1}. Intrinsic
approaches to Sobolev spaces of sections, weak covariant derivatives,
and covariant differential operators have been developed in several
complementary settings
\cite{Guneysu,KohrNistor,BehzadanHolst,VelazquezSandoval}. In
particular, the framework introduced in \cite{VelazquezSandoval}
provides the analytic background adopted here: weak covariant
derivatives, geometric integration by parts, local-to-global norm
comparisons, density results, and Sobolev embeddings for sections. The present paper extends that intrinsic framework by developing a
boundary trace theory in the same setting and applying it to a nonlinear
variational problem on vector bundles.

\medskip

On a manifold with boundary, homogeneous Dirichlet conditions require a
trace theory compatible with the covariant definition of the Sobolev
spaces. Trace results for sections occur in the Hilbert-space and
Dirac-operator framework of Booß-Bavnbek and Wojciechowski
\cite{BoossWojciechowski}. Große and Schneider
\cite{GrosseSchneider} studied traces onto embedded submanifolds of
manifolds of bounded geometry without boundary, including vector-bundle-valued Sobolev
spaces, whereas Van Schaftingen and Winter
\cite{VanSchaftingenWinter} recently obtained a first-order
gauge-covariant trace theory on manifolds with boundary, with optimal
fractional boundary spaces and estimates involving the curvature of the
connection. These results address either operator-theoretic boundary
problems, traces onto submanifolds, or refined first-order fractional
trace spaces. Our purpose is complementary. Starting directly from weak
covariant derivatives, we construct the traces of
\(u,\nabla u,\ldots,\nabla^{m-1}u\) for every integer \(m\geq1\) and
\(1\leq p<\infty\), and prove the intrinsic kernel characterization
\(W^{m,p}_0(\textbf{E})=\ker\mathcal T_m\).

Since our goal is to obtain an intrinsic characterization of homogeneous
Dirichlet boundary conditions for the variational problem, an integer-order boundary
scale is sufficient. The proof is direct and local-to-global: it uses
boundary charts, local frames, extension by zero from half-balls,
locality of weak derivatives and traces, and the relation between
ordinary and covariant derivatives. To the best of our knowledge, such a direct intrinsic construction of
higher-order trace operators on vector bundles has not previously
appeared in the literature.

\medskip

The second aim of the paper is variational. Operators with
generalized \((p,q)\)-growth model the interaction of different
nonlinear diffusion regimes and have been studied in connection with
generalized reaction--diffusion equations, regularity, multiplicity,
concentration, and ground-state phenomena; see, for example,
\cite{chil,heli,AlvesFigueiredo,ambrep} and the references therein. A
particularly close recent Euclidean counterpart is the
Kirchhoff--Boussinesq problem studied by Tavares, Carlos, and Salirrosas
\cite{TavaresCarlosSalirrosas2025}, in which a biharmonic operator is
perturbed by a broad nonhomogeneous \(p\&q\)-type term. At the level of
this quasilinear perturbation, their operator is structurally close to
the scalar version of the divergence-form operator considered here,
although the complete equation is fourth order and belongs to a
different variational setting. 

\medskip

Most of this literature concerns scalar functions on Euclidean domains
or on \(\mathbb R^N\). Although nonlinear scalar analysis on
Riemannian manifolds is well developed \cite{Aubin,Hebey1}, the passage
to sections of an arbitrary vector bundle is not formal: the principal
operator must be defined through a connection, the nonlinear term must
act fiberwise, and compactness and homogeneous boundary conditions must
be formulated intrinsically in spaces of sections. The framework below
extends the usual scalar \((p,q)\)-models to finite-rank Riemannian
vector bundles and allows nonlinearities generated by fiberwise
potentials. To the best of our knowledge, the corresponding
ground-state problem has not previously been treated in this bundle
setting with the present degree of generality.

\medskip

The geometric operator underlying the problem is also of independent
interest. When \(p=q=2\) and \(a\equiv1\), it reduces to the connection,
or Bochner, Laplacian
\(\bigl(\nabla^{\mathbf E}\bigr)^*\nabla^{\mathbf E}\). This operator
appears as the principal second-order part of several fundamental
geometric operators. Indeed, on differential \(k\)-forms, the
Weitzenböck formula takes the form
\[
\Delta_H
=
d\delta+\delta d
=
\bigl(\nabla^{\Lambda^k}\bigr)^*\nabla^{\Lambda^k}
+
\mathscr R_k,
\]
where \(\nabla^{\Lambda^k}\) is the connection on
\(\Lambda^kT^*M\) induced by the Levi--Civita connection and
\(\mathscr R_k\) is a curvature endomorphism.

To make the Dirac example explicit, suppose only in this paragraph that
\(M\) is oriented and spin. A spin structure consists of a principal
\(\operatorname{Spin}(n)\)-bundle
\(P_{\operatorname{Spin}}(M)\to M\) together with a two-sheeted bundle
map $\vartheta:P_{\operatorname{Spin}}(M)
\longrightarrow P_{\operatorname{SO}}(M)$ which is equivariant with respect to the double covering
\(\operatorname{Spin}(n)\to\operatorname{SO}(n)\). If
\(\rho:\operatorname{Spin}(n)\to\operatorname{GL}(\Sigma_n)\) is the
complex spin representation, the associated spinor bundle is $S:=P_{\operatorname{Spin}}(M)\times_{\rho}\Sigma_n$. The Levi--Civita connection induces the spin connection \(\nabla^S\),
and the spin Dirac operator is locally given by $D\psi
:=\sum_{j=1}^{n}e_j\mathbin{\cdot}\nabla^S_{e_j}\psi$, where \((e_j)_{j=1}^{n}\) is a local orthonormal frame and
\(\mathbin{\cdot}\) denotes Clifford multiplication. Its square
satisfies the Lichnerowicz formula 
\begin{eqnarray*}
    D^2
=
(\nabla^S)^*\nabla^S
+
\frac{1}{4}\operatorname{Scal}_g\operatorname{id}_S.
\end{eqnarray*}
More generally, the square of a Dirac-type operator on a Clifford module
has the form \(D^2=(\nabla^S)^*\nabla^S+\mathscr R_S\), for a suitable
curvature endomorphism \(\mathscr R_S\). Thus the connection Laplacian
is a basic building block of geometric analysis and mathematical
physics \cite{BoossWojciechowski,Guneysu}. These linear identities also provide useful precedents for nonlinear
geometric operators on particular vector bundles. For a differential
form \(\omega\in\Omega^k(M)\), the quadratic Hodge energy is
\[
E_2^H(\omega)
:=
\frac{1}{2}
\int_M
\bigl(|d\omega|^2+|\delta\omega|^2\bigr)\,d\lambda_g.
\]
The two terms arise from the two adjoint first-order components of the
Hodge--Dirac operator \(D_H=d+\delta\). For variations compactly
supported in \(\operatorname{Int}(M)\), or under boundary conditions
for which the boundary terms vanish, one has
\[
\left.\frac{d}{dt}\right|_{t=0}
E_2^H(\omega+t\eta)
=
\int_M
\langle\delta d\omega+d\delta\omega,\eta\rangle
\,d\lambda_g,
\]
so its Euler--Lagrange operator is the Hodge Laplacian
\(\Delta_H=d\delta+\delta d\). Replacing each quadratic contribution
by its \(p\)-growth analogue gives
\[
E_p^H(\omega)
:=
\frac{1}{p}
\int_M
\bigl(|d\omega|^p+|\delta\omega|^p\bigr)\,d\lambda_g,
\]
whose Euler--Lagrange operator is the \(p\)-Hodge Laplacian, also called
the \(p\)-Laplacian on differential forms,
\[
\Delta_{p,H}\omega
:=
\delta\bigl(|d\omega|^{p-2}d\omega\bigr)
+
d\bigl(|\delta\omega|^{p-2}\delta\omega\bigr).
\]
This operator was introduced and studied on closed manifolds by Seto
\cite{Seto}. It reduces to \(\Delta_H\) when \(p=2\), while on
\(0\)-forms it becomes, up to the sign convention, the scalar
\(p\)-Laplace--Beltrami operator. In a different direction, nonlinear
Dirac-type equations have been studied through \(p\)-Dirac operators in
Euclidean and spin settings \cite{NolderRyan}.

The present equation follows a distinct connection-based route. Rather
than being tied to the exterior differential, the codifferential, or
Clifford multiplication, it is defined directly through a compatible
connection on an arbitrary finite-rank Riemannian vector bundle. In
particular, when \(a\equiv1\) and \(p=q\), its principal part is $\nabla^*
\bigl(
|\nabla^{\mathbf E}u|^{p-2}\nabla^{\mathbf E}u
\bigr)
=
-\Delta_p^{\mathbf E}u$, where
\(\Delta_p^{\mathbf E}u:=-\nabla^*(|\nabla^{\mathbf E}u|^{p-2}
\nabla^{\mathbf E}u)\) denotes the Bochner \(p\)-Laplacian associated
with \(\nabla^{\mathbf E}\). Although the linear Hodge and Dirac
operators are related to the connection Laplacian by Weitzenböck
formulas, their nonlinear \(p\)-analogues do not in general coincide
with the Bochner \(p\)-Laplacian when \(p\neq2\). They should instead
be viewed as complementary nonlinear extensions associated with
different first-order geometric structures.

In the scalar case, the class considered here also contains sums of
\(p\)- and \(q\)-Laplacians with a prescribed-mean-curvature-type
contribution. More precisely, for \(p=2\), the term 
\begin{eqnarray*}
    \nabla^*\left(
\dis\frac{\nabla u}{\sqrt{1+|\nabla u|^2}}
\right)
\end{eqnarray*}
is, up to the sign convention, the mean-curvature operator for graphs
\cite{GiustiPMC}.

\medskip

 Although variational methods for quasilinear elliptic equations with
generalized \((p,q)\)-growth have been extensively developed in the
Euclidean setting (see, e.g.,
\cite{AlvesFigueiredo,Marcellini1991,Lieberman1991,MaranoMosconi2018, Costa2023, CostaFigueiredo2023, CostaFigueiredoJunior2023
}
and the references therein), their intrinsic counterparts on
Riemannian vector bundles remain largely unexplored. In particular, to
the best of our knowledge, no general trace theory for Sobolev sections
adapted to variational problems has previously been available. The
present paper addresses this gap by establishing such a theory and
showing that it provides the natural variational framework for studying
quasilinear elliptic equations with generalized \((p,q)\)-growth on
Riemannian vector bundles.

\medskip

Combining the intrinsic analytical framework for Sobolev sections on
Riemannian vector bundles \cite{VelazquezSandoval,GrosseSchneider,KohrNistor}
with variational techniques for quasilinear elliptic equations with generalized \((p,q)\)-growth
\cite{AlvesFigueiredo,ambrep,Marcellini1991,Lieberman1991,MaranoMosconi2018, Costa2023, CostaFigueiredo2023, CostaFigueiredoJunior2023}, we are led to consider the following generalized \((p,q)\)-equation on a Riemannian vector bundle:
\begin{equation}\label{ProblemP}
\tag{$P$}
\left\{
\begin{array}{ll}
\nabla^{*}\!\left(a(|\nabla^\textbf{E} u|_{g\otimes h_\textbf{E}}^{p})\,|\nabla^\textbf{E} u|_{g\otimes h_\textbf{E}}^{p-2}\nabla^\textbf{E} u\right)
=
f(u)
& \text{in } M,\\[6pt]
u=0 & \text{on } \partial M.
\end{array}
\right.
\end{equation}
Here \(\textbf{E}\to M\) is a smooth vector bundle of finite rank over a compact and connected \(n\)-dimensional Riemannian manifold \((M,g)\) with nonempty smooth boundary, \(h_\textbf{E}\) is a fiber metric on \(\textbf{E}\), \(|\cdot|_{g\otimes h_\textbf{E}}\) is the induced metric on tensor bundles with values in \(\textbf{E}\), and \(\nabla^\textbf{E}\) is a compatible connection. The symbol \(\nabla^*\) denotes the formal adjoint of the induced connection acting on \(\textbf{E}\)-valued one-forms. Throughout the paper we assume \(2\leq p\leq q<n\), and we write \( q^*:=\frac{nq}{n-q}\).

The hypotheses on \(a:[0,+\infty)\to[0,+\infty)\) are the following:
\begin{enumerate}[label=\((a_{\arabic*})\),ref=\((a_{\arabic*})\)]
\item \label{Bi2a_1}
The function \(a\) is continuous, and there exist constants
\(k_1,k_2,k_3,k_4\geq0\), with \(k_1>0\) and, in the case \(p<q\), \(k_2>0\), such that
\[
k_1t^p+k_2t^q
\leq
a(t^p)t^p
\leq
k_3t^p+k_4t^q
\]
for every \(t>0\).

\item \label{Bi2a_2}
The map \(\displaystyle t\mapsto \frac{a(t^p)}{t^{q-p}}\) is nonincreasing on \((0,+\infty)\).

\item \label{Bi2a_3}
If \(\displaystyle A(t):=\int_0^t a(s)\,ds\), then the function \(t\mapsto A(t^p)\) is convex on \([0,+\infty)\).
\end{enumerate}

Assumption~\ref{Bi2a_2} implies
\begin{equation}\label{eq:A_lower_from_monotonicity}
A(\rho^p)
\geq
\frac{p}{q}a(\rho^p)\rho^p
\qquad\text{for every }\rho\geq0.
\end{equation}
Indeed, for \(0<\tau\leq\rho\),
\(\displaystyle a(\tau^p)\geq a(\rho^p)(\tau/\rho)^{q-p}\), and hence
\[
A(\rho^p)
=
p\int_0^\rho a(\tau^p)\tau^{p-1}\,d\tau
\geq
p\,a(\rho^p)\rho^{p-q}\int_0^\rho\tau^{q-1}\,d\tau
=
\frac{p}{q}a(\rho^p)\rho^p.
\]

Let \(\mathcal F:\textbf{E}\to\mathbb R\) be a fiberwise potential. For each
\(x\in M\), let \(\mathcal F_x:\textbf{E}_x\to\mathbb R\) denote its restriction to the fiber \(\textbf{E}_x\). We assume that \(\mathcal F_x\) is of class \(C^1\) with respect to the fiber variable, and we denote by \(f_x:\textbf{E}_x\to \textbf{E}_x\) its gradient with respect to the fiber metric, characterized by
\begin{equation}
D\mathcal F_x(\xi)[\zeta]=\langle f_x(\xi),\zeta\rangle_{h_\textbf{E}},
\end{equation}
for every \(x\in M\) and every \(\xi,\zeta\in \textbf{E}_x\). The resulting fiber-preserving map \(f:\textbf{E}\to \textbf{E}\), defined by \(f(\xi):=f_{\pi(\xi)}(\xi)\), is assumed to be continuous. For a measurable section \(u\) of \(\textbf{E}\), we write \(f(u)\) for the measurable section defined almost everywhere by \(f(u)(x):=f_x(u(x))\).

The nonlinear term satisfies the following assumptions:
\begin{enumerate}[label=\((f_{\arabic*})\),ref=\((f_{\arabic*})\)]
\item \label{Bi2f_1}
The following limit holds uniformly with respect to \(x\in M\):
\[
\lim_{|\xi|_{h_\textbf{E}}\to0}
\frac{|f_x(\xi)|_{h_\textbf{E}}}{|\xi|_{h_\textbf{E}}^{p-1}}
=0.
\]

\item \label{Bi2f_2}
There exist constants \(C_f>0\) and \(r\in(q,q^*)\) such that
\[
|f_x(\xi)|_{h_\textbf{E}}
\leq
C_f\bigl(1+|\xi|_{h_\textbf{E}}^{r-1}\bigr)
\]
for every \(x\in M\) and every \(\xi\in \textbf{E}_x\).

\item \label{Bi2f_3}
There exists \(\theta\in(q,r]\) such that
\[
0<\theta\mathcal F_x(\xi)
\leq
\langle f_x(\xi),\xi\rangle_{h_\textbf{E}}
\]
for every \(x\in M\) and every \(\xi\in \textbf{E}_x\setminus\{0_x\}\).

\item \label{Bi2f_4}
For every \(x\in M\) and every \(\xi\in \textbf{E}_x\setminus\{0_x\}\), the map
\[
t\longmapsto
\frac{\langle f_x(t\xi),t\xi\rangle_{h_\textbf{E}}}{t^q}
\]
is strictly increasing on \((0,+\infty)\).
\end{enumerate}
Since each \(\mathcal F_x\) is continuous, assumptions~\ref{Bi2f_1} and \ref{Bi2f_3} imply that \(\mathcal F_x(0_x)=0\) for every \(x\in M\). Consequently, $\mathcal F_x(\xi)
=
\int_0^1\langle f_x(t\xi),\xi\rangle_{h_\textbf{E}}\,dt$. The map $[0,1]\times \textbf{E}\longrightarrow\mathbb R$, $(t,\xi)\longmapsto
\langle f(t\xi),\xi\rangle_{h_\textbf{E}}$, is continuous. Since the interval \([0,1]\) is compact, integration with
respect to \(t\) shows that the resulting map
\(\mathcal F:\textbf{E}\to\mathbb R\) is continuous.

\bigskip

Under the assumptions above we prove the existence of a ground state
for \eqref{ProblemP}. Since the energy functional is only of class
\(C^1\), we avoid treating the Nehari set as a differentiable
hypersurface and instead verify the abstract radial-reduction
assumptions of Szulkin and Weth
\cite{SzulkinWethNehari}. The resulting projection identifies the
Nehari set homeomorphically with the unit sphere, while compact Sobolev
embeddings for sections provide the compactness needed to attain the
ground-state level.

\bigskip

Our work is closely related to
\cite{AlvesFigueiredo,ambrep,VelazquezSandoval,Marcellini1991,Lieberman1991,MaranoMosconi2018,Costa2023,CostaFigueiredo2023,CostaFigueiredoJunior2023},
which provide relevant foundations in intrinsic Sobolev theory, quasilinear elliptic equations, and variational methods.
The main contributions of the present paper are summarized as follows.

\begin{itemize}

\item[$(i)$]
In contrast to the Euclidean framework considered in
\cite{AlvesFigueiredo}, where the unknown is a scalar function and
the operator is formulated in divergence form, we develop an intrinsic
higher-order Sobolev and trace framework for sections of finite-rank
Riemannian vector bundles over compact Riemannian manifolds with smooth
boundary. Building on the Sobolev theory developed in
\cite{VelazquezSandoval}, we introduce higher-order trace operators
associated with weak covariant derivatives and characterize Sobolev
spaces with homogeneous Dirichlet boundary conditions through the
kernels of these operators.

\item[$(ii)$]
As an application of this framework, we consider the first-order case
$m=1$ and study quasilinear elliptic equations with generalized
$(p,q)$-growth on Riemannian vector bundles. We establish the existence
of ground state solutions by variational methods, thereby extending the
scope of the variational approaches developed in
\cite{Costa2023,CostaFigueiredo2023,CostaFigueiredoJunior2023}
to the intrinsic geometric setting.

\item[$(iii)$]
The transition from Euclidean domains to Riemannian vector bundles
requires an intrinsic reformulation of the underlying analytical
arguments. In particular, the absence of a global trivialization leads
naturally to the use of weak covariant derivatives, induced tensorial
connections, and geometric integration by parts. The resulting
framework provides a geometric setting in which quasilinear elliptic
problems can be formulated and treated intrinsically on Riemannian
vector bundles.

\end{itemize}

This paper is organized as follows. Section
\ref{sec:main_results_examples} states the main variational result and
presents representative examples of the operators covered by the
framework. Section \ref{sec:notation_preliminaries} fixes the geometric
notation and recalls the Sobolev spaces of sections defined through weak
covariant derivatives. Section \ref{sec:trace_theory} develops the
intrinsic trace theory and characterizes the zero-trace Sobolev spaces.
Section \ref{sec:variational_framework} introduces the energy and the
Nehari set, verifies the radial-reduction hypotheses of Szulkin and Weth
\cite{SzulkinWethNehari}, and proves the ground-state theorem. Finally,
Section \ref{sec:appendix} collects the superposition and
differentiability results used in the variational arguments.

\section{Main results and examples}\label{sec:main_results_examples}

The first group of results concerns the intrinsic trace theory for
Sobolev sections. Theorems
\ref{thm:integer_order_trace_vector_bundle} and
\ref{thm:higher_order_traces_bundle} construct the traces of a section
and of its covariant derivatives, while Theorem
\ref{thm:zero_trace_characterization_order_m_bundle} identifies the
zero-trace space with the kernel of the resulting boundary operator.
We now formulate the variational result and illustrate the class of
operators covered by it.

We first introduce the notion of weak solution used in the variational part. A section \(u\in W^{1,q}_0(\textbf{E})\) is said to be a weak solution of \eqref{ProblemP} if, for every \(\phi\in W^{1,q}_0(\textbf{E})\), one has
\begin{equation*}
\int_M a(|\nabla^\textbf{E} u|_{g\otimes h_\textbf{E}}^p)|\nabla^\textbf{E} u|_{g\otimes h_\textbf{E}}^{p-2}\langle\nabla^\textbf{E} u,\nabla^\textbf{E}\phi\rangle_{g\otimes h_\textbf{E}}\,d\lambda_g
=
\int_M \langle f(u),\phi\rangle_{h_\textbf{E}}\,d\lambda_g.
\end{equation*}
Once the energy functional is introduced, a ground state solution will mean a nontrivial weak solution whose energy is minimal among all nontrivial weak solutions.

We now state the main variational result of this manuscript.

\begin{theorem}\label{teo_A}
Assume that \ref{Bi2a_1}--\ref{Bi2a_3} and
\ref{Bi2f_1}--\ref{Bi2f_4} hold. Then problem
\eqref{ProblemP} admits a ground state solution.
\end{theorem}

To establish this result, we will develop the necessary theory in the subsequent sections. Before presenting the formal proof, however, we provide several examples to illustrate how familiar quasilinear equations fit within our framework. 

\begin{remark}
 Throughout these examples, we consider the model potential \(\displaystyle \mathcal F_x(\xi):=\frac{1}{r}|\xi|_{h_\textbf{E}}^r\), whose fiber gradient is \(f_x(\xi)=|\xi|_{h_\textbf{E}}^{r-2}\xi\). For \(r\in(q,q^*)\), this nonlinearity satisfies \ref{Bi2f_1}, \ref{Bi2f_2}, \ref{Bi2f_3}, and \ref{Bi2f_4} for every \(\theta\in(q,r]\). Recall that \(\Delta_p^\textbf{E} u:=-\nabla^*(|\nabla^\textbf{E} u|_{g\otimes h_\textbf{E}}^{p-2}\nabla^\textbf{E} u)\) denotes the Bochner \(p\)-Laplacian associated with the connection \(\nabla^\textbf{E}\).   
\end{remark}

\begin{example}
Consider the constant function $a(t)\equiv1$, $t \geq 0$. If, in addition, \(p=q\), then the operator in \eqref{ProblemP}
coincides with the \(p\)-Laplacian. Hence, Theorem~\ref{teo_A} is valid for the problem
\[
\begin{cases}
-\Delta_p^{\mathbf E}u=f(u), & \text{in } M,\\
u=0, & \text{on } \partial M.
\end{cases}
\]
\end{example}

\begin{example}
Consider the function \(a(t)=1+t^{\frac{q-p}{p}}\), $t \geq 0$. It is straightforward to verify that \(a\) satisfies assumptions~\ref{Bi2a_1}-\ref{Bi2a_3} with $k_1=k_2=k_3=k_4=1$. Therefore, Theorem~\ref{teo_A} is applicable to the following quasilinear problem:
\begin{equation*}
\left\{
\begin{array}{rcl}
-\Delta_p^\textbf{E} u-\Delta_q^\textbf{E} u&=&f(u)\quad\text{in }M,\\
u&=&0\quad\text{on }\partial M.
\end{array}
\right.
\end{equation*}
\end{example}

\begin{example}\rm
Consider the function \(a(t)=1+\displaystyle\frac{1}{(1+t)^{\frac{p-1}{p}}}\), $t \geq 0$ and \(q=p\). One readily verifies that \(a\) satisfies assumptions
\ref{Bi2a_1}--\ref{Bi2a_3} with $k_1=1, k_2=0, k_3=2, k_4=0$.
Therefore, Theorem~\ref{teo_A} is applicable to the following quasilinear problem:
\begin{equation*}
\left\{
\begin{array}{rcl}
-\Delta_{p}^{\textbf{E}}u+\nabla^*\left(
\displaystyle\frac{|\nabla^\textbf{E} u|_{g\otimes h_\textbf{E}}^{p-2}\nabla^\textbf{E} u}{(1+|\nabla^\textbf{E} u|_{g\otimes h_\textbf{E}}^p)^{\frac{p-1}{p}}}\right)
&=&f(u)\quad\text{in }M,\\
u&=&0\quad\text{on }\partial M.
\end{array}
\right.
\end{equation*}
\end{example}

\begin{example}
Consider the function $a(t)=1+t^{\frac{q-p}{p}}+\frac{1}{(1+t)^{\frac{p-2}{p}}}$, $t\geq 0$. One readily verifies that \(a\) satisfies assumptions
\ref{Bi2a_1}--\ref{Bi2a_3} with $k_1=k_2=k_3=1, k_4=2$.
Therefore, Theorem~\ref{teo_A} is applicable to the following quasilinear problem:
\begin{equation*}
\left\{
\begin{array}{rcl}
-\Delta_p^\textbf{E} u-\Delta_q^\textbf{E} u+
\nabla^*\left(\displaystyle\frac{|\nabla^\textbf{E} u|_{g\otimes h_\textbf{E}}^{p-2}\nabla^\textbf{E} u}
{(1+|\nabla^\textbf{E} u|_{g\otimes h_\textbf{E}}^p)^{\frac{p-1}{p}}}\right)
&=&f(u)\quad\text{in }M,\\
u&=&0\quad\text{on }\partial M.
\end{array}
\right.
\end{equation*}
\end{example}

\begin{remark}
    Observe that, in each of the preceding examples, when
\(\mathbf{E}=M\times\mathbb{R}\) is the trivial line bundle,
\(\Delta_p^{\mathbf E}\) reduces to the classical \(p\)-Laplace--Beltrami
operator \(\Delta_{p,g}\). Therefore, the corresponding scalar quasilinear
problems on compact Riemannian manifolds with boundary are encompassed by our
vector bundle setting.
\end{remark}

\begin{remark}
The vector-bundle setting requires two genuinely fiberwise
reformulations. First, a scalar nonlinearity
\(f:\mathbb R\to\mathbb R\) cannot be applied directly to a section
\(u:M\to\textbf{E}\); this is why the nonlinear term is formulated
through a fiberwise potential \(\mathcal F:\textbf{E}\to\mathbb R\)
and its fiber gradient. Second, although the assumptions on \(a\) remain
essentially unchanged, the monotonicity assumptions on \(f\) must be
expressed along the lines \(\mathbb R\xi\subset\textbf{E}_x\), for each
nonzero \(\xi\in\textbf{E}_x\).
\end{remark}

%%%%%%%%%%%%%%%%%%%%%%%%%%%%%%%%%%%%%%%%%%%%%%%%%%%%%%%%%%
%%%%%%%%%%%%%%%%%%%%%%%%%%%%%%%%%%%%%%%%%%%%%%%%%%%%%%%%%%
%%%%%%%%%%%%%%%%%%%%%%%%%%%%%%%%%%%%%%%%%%%%%%%%%%%%%%%%%%

%%%%%%%%%%%%%%%%%%%%%%%%%%%%%%%%%%%%%%%%%%%%%%%%%%%%%%%%%%
%%%%%%%%%%%%%%%%%%%%%%%%%%%%%%%%%%%%%%%%%%%%%%%%%%%%%%%%%%
%%%%%%%%%%%%%%%%%%%%%%%%%%%%%%%%%%%%%%%%%%%%%%%%%%%%%%%%%%
    
\section{Notation and preliminaries}\label{sec:notation_preliminaries}

In this section we fix the geometric notation used throughout the paper and recall the Sobolev spaces of sections needed in the variational setting.

Throughout the paper, \((M,g)\) denotes a compact \(n\)-dimensional Riemannian manifold with smooth boundary \(\partial M\), and \(\operatorname{Int}(M)\) denotes its interior. We write \(d\lambda_g\) for the Riemannian measure induced by \(g\); this measure is well defined even when \(M\) is not orientable. We assume that \(\textbf{\textbf{E}}\to M\) is a smooth vector bundle of finite rank endowed with a fiber metric \(h_\textbf{E}\) and a compatible connection \(\nabla^\textbf{E}\). The induced metric on tensor bundles with values in \(\textbf{E}\), such as \(T^{(k,l)}(TM)\otimes \textbf{E}\), will be denoted by \(\langle \cdot,\cdot\rangle_{g\otimes h_\textbf{E}}\). In local frame coordinates, it is given by
\begin{align*}
\langle F,G\rangle_{g\otimes h_\textbf{E}}
 = g_{i_1r_1}\cdots g_{i_kr_k}g^{j_1s_1}\cdots g^{j_ls_l}h_{ab}
F^{i_1\dots i_k a}_{j_1\dots j_l}G^{r_1\dots r_k b}_{s_1\dots s_l}, \quad \forall F,G \in T^{(k,l)}(TM)\otimes \textbf{E}.
\end{align*}

\noindent We denote by \(\Gamma(\textbf{E})\) the space of smooth sections of \(\textbf{E}\). If \(u\in\Gamma(\textbf{E})\), its support is \(\operatorname{supp}(u):=\overline{\{x\in M\mid u(x)\neq0\}}\). We write 
\begin{equation*}
    \Gamma_c(\textbf{E}):=\{u\in\Gamma(\textbf{E})\mid \operatorname{supp}(u)\Subset M\}
\end{equation*}
and 
\begin{equation*}
    \Gamma_{c,\operatorname{Int}(M)}(\textbf{E}):=\{u\in\Gamma(\textbf{E})\mid \operatorname{supp}(u)\Subset\operatorname{Int}(M)\}.
\end{equation*}
Notice that a section compactly supported in \(M\) need not be compactly supported in \(\operatorname{Int}(M)\); hence \(\Gamma_{c,\operatorname{Int}(M)}(\textbf{E})\subsetneq\Gamma_c(\textbf{E})\) whenever the boundary is nonempty. We shall use the following definition of the spaces
\(L^p_{\mathrm{loc}}(\mathbf E)\) and \(L^p(\mathbf E)\).

\begin{definition}\label{def:Lp_sections}
Let $(M,g)$ be a Riemannian manifold, let $1\leq p\leq\infty$, and let
$\pi:\mathbf E\to M$ be a smooth vector bundle equipped with a fiber
metric $h_{\mathbf E}$.

A map $u:M\to\mathbf E$ is called a measurable section if
\begin{enumerate}[label=(\alph*)]
\item $u$ is measurable;
\item $(\pi\circ u)(x)=x$ for almost every $x\in M$.
\end{enumerate}

A measurable section $u$ is called an $L^p$-section if
$|u|_{h_{\mathbf E}}\in L^p(M,d\lambda_g)$.

We denote by
\[
\widetilde L^p(\mathbf E)
:=
\left\{
u:M\to\mathbf E
\mid
u\text{ is an }L^p\text{-section}
\right\}.
\]

Two elements $u,v\in\widetilde L^p(\mathbf E)$ are identified whenever
\[
u=v
\quad\text{almost everywhere on }M.
\]

The space of $L^p$-sections is then defined by
\[
L^p(\mathbf E)
:=
\widetilde L^p(\mathbf E)/\!\sim.
\]

Similarly,
\[
L^p_{\mathrm{loc}}(\mathbf E)
:=
\left\{
u:M\to\mathbf E
\mid
u\text{ is measurable},
\ (\pi\circ u)(x)=x\text{ a.e.},
\ |u|_{h_{\mathbf E}}\in
L^p_{\mathrm{loc}}(M,d\lambda_g)
\right\}\Big/\!\sim.
\]
\end{definition}

 Although the connection on \(\textbf{E}\) is denoted by \(\nabla^\textbf{E}\), the same symbol \(\nabla\) will be used for the induced connections on the tensor bundles \(T^{(0,s)}(TM)\otimes \textbf{E}\). Thus \(\nabla^s u\) denotes the \(s\)-th covariant derivative of \(u\), viewed as a section of \(T^{(0,s)}(TM)\otimes \textbf{E}\). Schematically, the iteration is given by

\[\begin{tikzcd}[ampersand replacement=\&, column sep=5em, row sep=2em]
  \Gamma(\textbf{E}) \arrow[r, "\nabla^\textbf{E}"] \& 
  \Gamma\bigl(T^{(0,1)}(TM)\otimes \textbf{E}\bigr) 
  \arrow[d, out=0, in=0, looseness=1.4, "{\nabla^{T^{(0,1)}(TM)\otimes \textbf{E}}}" description] \\
  \Gamma\bigl(T^{(0,3)}(TM)\otimes \textbf{E}\bigr) 
  \arrow[d, out=180, in=180, looseness=1.4, "{\nabla^{T^{(0,3)}(TM)\otimes \textbf{E}}}" description] \& 
  \Gamma\bigl(T^{(0,2)}(TM)\otimes \textbf{E}\bigr) \arrow[l, "{\nabla^{T^{(0,2)}(TM)\otimes \textbf{E}}}"] \\
  \cdots \arrow[r, "{\nabla^{T^{(0,s-1)}(TM)\otimes \textbf{E}}}"] \& 
  \Gamma\bigl(T^{(0,s)}(TM)\otimes \textbf{E}\bigr).
\end{tikzcd}
\]

We now recall the weak covariant derivatives used to define Sobolev spaces of sections.
\begin{definition}\label{def:weak_deriv}

Let \(s\geq1\). The \emph{formal adjoint} of the iterated covariant derivative \(\nabla^s\), denoted by \((\nabla^s)^*\), is the differential operator characterized by
\begin{equation}\label{eq:adjoint_def}
\int_M \langle \nabla^s u,v\rangle_{g\otimes h_\textbf{E}}\,d\lambda_g
=
\int_M \langle u,(\nabla^s)^*v\rangle_{h_\textbf{E}}\,d\lambda_g,
\end{equation}
for all \(u\in\Gamma_c(\textbf{E})\) and \(v\in\Gamma_{c,\operatorname{Int}(M)}(T^{(0,s)}(TM)\otimes \textbf{E})\). A section \(u\in L^1_{\operatorname{loc}}(\textbf{E})\) is said to admit a \emph{weak covariant derivative of order \(s\)} if there exists \(v\in L^1_{\operatorname{loc}}(T^{(0,s)}(TM)\otimes \textbf{E})\) such that 
\begin{equation}\label{adjoin_important}
\displaystyle \int_M\langle v,\psi\rangle_{g\otimes h_\textbf{E}}\,d\lambda_g=\int_M\langle u,(\nabla^s)^*\psi\rangle_{h_\textbf{E}}\,d\lambda_g
\end{equation}
for every \(\psi\in\Gamma_{c,\operatorname{Int}(M)}(T^{(0,s)}(TM)\otimes \textbf{E})\). Such a section \(v\), if it exists, is unique almost everywhere and is denoted by \(\nabla_w^s u\).
\end{definition}

\begin{definition}
    For \(m\in\mathbb N\) and \(1\leq p<\infty\), we define
\[
W^{m,p}(\textbf{E}):=\left\{ u\in L^{p}(\textbf{E}) \mathrel{\Biggr|} \begin{aligned} &\text{for all } 1\leq s\leq m, \text{ the weak derivative } \nabla_w^s u \\ &\text{exists and } \nabla_w^{s}u \in L^{p}(T^{(0,s)}(TM)\otimes \textbf{E}) \end{aligned} \right\}.\]

We endow \(W^{m,p}(\textbf{E})\) with the norm
\[
\|u\|_{W^{m,p}(\textbf{E})}
:=
\left(
\sum_{s=0}^{m}
\int_M
|\nabla_w^s u|_{g\otimes h_\textbf{E}}^{p}\,
d\lambda_g
\right)^{1/p},
\]
with the convention \(\nabla_w^0u=u\).
\end{definition}

\section{Trace theory and zero-trace Sobolev spaces in vector bundles}\label{sec:trace_theory}

To restrict Sobolev sections to the boundary, we need the trace theory associated with manifolds with boundary. Since we shall work with sections in vector bundles rather than scalar functions, we recall the construction in a form adapted to our setting.  We recall the following classical trace theorem in the half-space; see,
for instance, \cite{Leoni}.
\begin{theorem}[Trace theorem in the half-space]\label{thm:trace_half_space_order_m}
Let \(m\geq1\) and \(1\leq p<\infty\). For every multi-index \(\beta\) with \(|\beta|\leq m-1\), the map \(\gamma_{0}D^{\beta}:C_c^\infty(\overline{\mathbb R^n_+})\longrightarrow C^{\infty}(\mathbb{R}^{n-1})\) given by \(\gamma_{0}D^{\beta}(f)=(D^\beta f)|_{\partial\mathbb R^n_+}\), admits a unique continuous linear extension
\[
\gamma_0D^\beta:
W^{m,p}(\mathbb R^n_+)
\longrightarrow
W^{m-1-|\beta|,p}(\mathbb R^{n-1}).
\]
Moreover,
\[
W^{m,p}_0(\mathbb R^n_+)
=
\left\{
f\in W^{m,p}(\mathbb R^n_+)
\mathrel{\Big|}
\gamma_0D^\beta f=0
\text{ for every }|\beta|\leq m-1
\right\},
\]
where \(W^{m,p}_0(\mathbb R^n_+)\) denotes the closure of \(C_c^\infty(\mathbb R^n_+)\) in \(W^{m,p}(\mathbb R^n_+)\).
\end{theorem}
\begin{remark}
For our purposes, it is more convenient to formulate the trace theorem
in terms of all partial derivatives rather than only normal
derivatives, since this admits a natural intrinsic extension to vector
bundles, where boundary data are described by covariant jets associated
with the ambient connection. In the Euclidean setting, both formulations
are equivalent. Indeed, since the outward unit normal to
\(\partial\mathbb R^n_+\) is, up to sign, the constant vector \(e_n\),
one has \(\partial_\nu=\pm\partial_n\). Moreover, by density and the
continuity of the trace operator, the trace commutes with tangential
differentiation, that is,
\[
\gamma_0D^\beta u
=
D_{x'}^{\beta'}
\bigl(\gamma_0\partial_n^{\beta_n}u\bigr),
\qquad \text{where}\ 
\beta=(\beta',\beta_n).
\]
Therefore,
\(\gamma_0D^\beta u=0\) for every \(|\beta|\le m-1\) if and only if
\(\gamma_0\partial_\nu^j u=0\) for every \(0\le j\le m-1\).
\end{remark}
In boundary charts, we do not work with the whole half-space, but with half-balls. The next lemma explains how to pass from a function defined in a half-ball to a function defined in the half-space. Traces on half-balls will always be taken after this extension, so that no independent trace operator on the half-ball is needed.
\begin{lemma}[Extension by zero from a half-ball]\label{lem:zero_extension_half_ball_order_m}
Let \(B^+=B_R(0)\cap\mathbb R^n_+\) and \(B'=B_R(0)\cap\partial\mathbb R^n_+\). Suppose that \(v\in W^{m,p}(B^+)\) and that there exists a compact set \(K\subset B_R(0)\cap\{x_n\geq0\}\), at positive distance from the spherical part \(\partial B_R(0)\cap\{x_n\geq0\}\), such that \(v=0\) almost everywhere in \(B^+\setminus K\). Define \(\widetilde v:\mathbb R^n_+\to\mathbb R\) by \(\widetilde v(x)=v(x)\) if \(x\in B^+\), and \(\widetilde v(x)=0\) if \(x\in\mathbb R^n_+\setminus B^+\). Then \(\widetilde v\in W^{m,p}(\mathbb R^n_+)\), and for every multi-index \(\beta\) with \(|\beta|\leq m\), \(D^\beta\widetilde v=\widetilde{D^\beta v}\) in the weak sense, where \(\widetilde{D^\beta v}\) denotes the extension to zero of \(D^\beta v\) in \(\mathbb R^n_+\). In particular, \(\|\widetilde v\|_{W^{m,p}(\mathbb R^n_+)}=\|v\|_{W^{m,p}(B^+)}\).
\end{lemma}

\begin{proof}
Let \(\Gamma_R:=\partial B_R(0)\cap\{x_n\geq0\}\) be the spherical part of the half-ball. If \(K=\varnothing\), then \(v=0\) almost everywhere in \(B^+\), and the conclusion is immediate. Suppose that \(K\neq\varnothing\). Since \(K\) is at positive distance from \(\Gamma_R\), let \(\delta:=\operatorname{dist}(K,\Gamma_R)>0\). For every \(x\in K\), one has \(\operatorname{dist}(x,\Gamma_R)=R-|x|\). Indeed, the inequality \(\operatorname{dist}(x,\Gamma_R)\leq R-|x|\) follows by taking the radial point \(\frac{R}{|x|}x\in\Gamma_R\) when \(x\neq0\), and the reverse inequality follows from \(|x-y|\geq ||y|-|x||=R-|x|\) for every \(y\in\Gamma_R\). The case \(x=0\) is immediate. Hence \(|x|\leq R-\delta\) for every \(x\in K\), and in particular \(K\subset B_{R-\frac{\delta}{2}}(0)\cap\{x_n\geq0\}\).

Define \(U:=B^+\cap\{x\in\mathbb R^n\mid |x|>R-\frac{\delta}{2}\}\). Then \(U\) is a relative neighborhood of \(\Gamma_R\) inside \(B_R(0)\cap\{x_n\geq0\}\), and \(U\cap K=\varnothing\). Since \(v=0\) almost everywhere in \(B^+\setminus K\), we have \(v=0\) almost everywhere in \(U\).

By locality of weak derivatives, \(D^\beta v=0\) almost everywhere in \(U\) for every \(|\beta|\leq m\). Indeed, if \(w\in W^{k,p}(B^+)\) vanishes almost everywhere in a relative open set \(V\subset B^+\), then, for every \(\psi\in C_c^\infty(V)\), regarded as a test function in \(B^+\), one has 
\begin{eqnarray*}
    \displaystyle\int_V D^\alpha w\,\psi\,dx=\int_{B^+}D^\alpha w\,\psi\,dx=(-1)^{|\alpha|}\int_{B^+}w\,D^\alpha\psi\,dx=0
\end{eqnarray*}
Since \(D^\alpha w\in L^p(V)\), it follows that \(D^\alpha w=0\) almost everywhere in \(V\).

We now prove that \(D^\beta\widetilde v=\widetilde{D^\beta v}\) weakly for every \(|\beta|\leq m\). Fix such a \(\beta\) and let \(\varphi\in C_c^\infty(\mathbb R^n_+)\). Choose \(\chi\in C_c^\infty(\mathbb R^n)\) such that \(\chi\equiv1\) on \(B_{R-\frac{\delta}{2}}(0)\) and \(\operatorname{supp}\chi\subset B_{R-\frac{\delta}{4}}(0)\). Then \((\chi|_{\mathbb R^n_+})\varphi|_{B^+}\in C_c^\infty(B^+)\), because \(\chi\) vanishes near the spherical part of \(B^+\). Since \(\widetilde v=0\) in \(\mathbb R^n_+\setminus B^+\), we have 
\begin{eqnarray*}
    \displaystyle\int_{\mathbb R^n_+}\widetilde v\,D^\beta\varphi\,dx=\int_{B^+}v\,D^\beta\varphi\,dx
\end{eqnarray*}
Moreover, \(v=0\) almost everywhere away from \(K\), and \(\chi\equiv1\) on a neighborhood of \(K\). Hence \(vD^\beta\varphi=vD^\beta(\chi\varphi)\) almost everywhere in \(B^+\). Therefore
\[
\displaystyle\int_{B^+}v\,D^\beta\varphi\,dx
=
\int_{B^+}v\,D^\beta(\chi\varphi)\,dx
=
(-1)^{|\beta|}
\int_{B^+}D^\beta v\,\chi\varphi\,dx.
\]
Inside \(B^+\), the set where \(\chi\neq1\) is contained in \(U\), up to the sphere \(\{|x|=R-\frac{\delta}{2}\}\), which has measure zero. Since \(D^\beta v=0\) almost everywhere in \(U\), we get \(\int_{B^+}D^\beta v\,\chi\varphi\,dx=\int_{B^+}D^\beta v\,\varphi\,dx=\int_{\mathbb R^n_+}\widetilde{D^\beta v}\,\varphi\,dx\). Combining the previous identities gives
\[
\int_{\mathbb R^n_+}\widetilde v\,D^\beta\varphi\,dx
=
(-1)^{|\beta|}
\int_{\mathbb R^n_+}\widetilde{D^\beta v}\,\varphi\,dx.
\]
Thus \(D^\beta\widetilde v=\widetilde{D^\beta v}\) weakly. Since this holds for all \(|\beta|\leq m\), \(\widetilde v\in W^{m,p}(\mathbb R^n_+)\), and the equality of norms follows directly from \(D^\beta\widetilde v=\widetilde{D^\beta v}\).
\end{proof}

We shall use the following convention: If \(v\in W^{m,p}(B^+)\) satisfies the support hypothesis of Lemma \ref{lem:zero_extension_half_ball_order_m}, and if \(\widetilde v\) denotes its extension by zero to \(\mathbb R^n_+\), then, for \(|\beta|\leq m-1\), writing \(\gamma_0D^\beta v=0\) on \(B'\) means that \(\left.(\gamma_0D^\beta\widetilde v)\right|_{B'}=0\) in \(W^{m-1-|\beta|,p}(B')\). This notation does not introduce an independent trace operator on the half-ball; it only abbreviates the half-space trace after extension by zero.

\begin{lemma}[Locality of the trace in the half-space]\label{lem:locality_trace_half_space_order_m}
Let \(O\subset\partial\mathbb R^n_+\) be relatively open, and let \(F\in W^{1,p}(\mathbb R^n_+)\). Suppose that there exists an open set \(V\subset\mathbb R^n\), with \(O\subset V\cap\partial\mathbb R^n_+\), such that \(F=0\) almost everywhere in \(V\cap\mathbb R^n_+\). Then \(\gamma_0F=0\) on \(O\).
\end{lemma}

\begin{proof}
Fix \(x_0\in O\). Since \(O\) is relatively open in \(\partial\mathbb R^n_+\) and \(V\) is open in \(\mathbb R^n\), we can choose \(r>0\) such that \(B_{2r}(x_0)\subset V\) and \(B_{2r}(x_0)\cap\partial\mathbb R^n_+\subset O\). Take \(\chi\in C_c^\infty(\mathbb R^n)\) such that \(\operatorname{supp}\chi\subset B_{2r}(x_0)\) and \(\chi\equiv1\) on \(B_r(x_0)\). Then \(\chi F=0\) in \(W^{1,p}(\mathbb R^n_+)\), because \(\operatorname{supp}\chi\cap\mathbb R^n_+\subset V\cap\mathbb R^n_+\) and \(F=0\) almost everywhere there. Hence \(\gamma_0(\chi F)=0\).

We claim that \(\gamma_0(\chi G)=\chi|_{\partial\mathbb R^n_+}\gamma_0G\) for every \(G\in W^{1,p}(\mathbb R^n_+)\). For \(G\in C_c^\infty(\overline{\mathbb R^n_+})\), the identity is immediate. For general \(G\), take \(G_j\in C_c^\infty(\overline{\mathbb R^n_+})\) such that \(G_j\to G\) in \(W^{1,p}(\mathbb R^n_+)\). Since multiplication by \(\chi|_{\mathbb R^n_+}\) is continuous in \(W^{1,p}(\mathbb R^n_+)\), we have \(\chi G_j\to\chi G\) in \(W^{1,p}(\mathbb R^n_+)\). By continuity of the trace, \(\gamma_0(\chi G_j)\to\gamma_0(\chi G)\) in \(L^p(\partial\mathbb R^n_+)\), and since multiplication by \(\chi|_{\partial\mathbb R^n_+}\) is continuous in \(L^p(\partial\mathbb R^n_+)\), we also have \(\chi|_{\partial\mathbb R^n_+}\gamma_0G_j\to\chi|_{\partial\mathbb R^n_+}\gamma_0G\). Passing to the limit in the smooth identity gives the claim.

Applying the claim to \(G=F\), we obtain \(0=\gamma_0(\chi F)=\chi|_{\partial\mathbb R^n_+}\gamma_0F\). Since \(\chi\equiv1\) on \(B_r(x_0)\), it follows that \(\gamma_0F=0\) on \(B_r(x_0)\cap\partial\mathbb R^n_+\). Since \(x_0\in O\) was arbitrary, \(\gamma_0F=0\) on \(O\).
\end{proof}

\begin{corollary}[Local characterization in a half-ball]\label{cor:local_characterization_half_ball_order_m}
Let \(B^+=B_R(0)\cap\mathbb R^n_+\), let \(B'=B_R(0)\cap\partial\mathbb R^n_+\), and let \(v\in W^{m,p}(B^+)\). Suppose that there exists a compact set \(K\subset B_R(0)\cap\{x_n\geq0\}\), at positive distance from the spherical part \(\partial B_R(0)\cap\{x_n\geq0\}\), such that \(v=0\) almost everywhere in \(B^+\setminus K\). If \(\gamma_0D^\beta v=0\) on \(B'\), in the sense of the convention above, for every \(|\beta|\leq m-1\), then \(v\) belongs to the closure of \(C_c^\infty(B^+)\) in \(W^{m,p}(B^+)\).
\end{corollary}

\begin{proof}
If \(K=\varnothing\), then \(v=0\) almost everywhere in \(B^+\), and the conclusion is immediate. Suppose that \(K\neq\varnothing\). Let \(\widetilde v\) be the extension by zero of \(v\) to \(\mathbb R^n_+\). By Lemma \ref{lem:zero_extension_half_ball_order_m}, \(\widetilde v\in W^{m,p}(\mathbb R^n_+)\) and \(D^\beta\widetilde v=\widetilde{D^\beta v}\) for every \(|\beta|\leq m\). By hypothesis, \(\left.(\gamma_0D^\beta\widetilde v)\right|_{B'}=0\) for every \(|\beta|\leq m-1\). We show that, in fact, \(\gamma_0D^\beta\widetilde v=0\) on all of \(\partial\mathbb R^n_+\) for every \(|\beta|\leq m-1\).

Fix \(|\beta|\leq m-1\). It remains to check the complement of \(B'\). Let \(\Gamma_R:=\partial B_R(0)\cap\{x_n\geq0\}\) and \(\delta:=\operatorname{dist}(K,\Gamma_R)>0\). As in the proof of Lemma \ref{lem:zero_extension_half_ball_order_m}, \(D^\alpha v=0\) almost everywhere in \(U:=B^+\cap\{x\in\mathbb R^n\mid |x|>R-\frac{\delta}{2}\}\) for every \(|\alpha|\leq m\).

Let \(x_0=(x_0',0)\in\partial\mathbb R^n_+\setminus B'\). Then \(|x_0'|\geq R\). If \(|x_0'|>R\), take \(\rho:=\frac{|x_0'|-R}{2}>0\). Then \(B_\rho(x_0)\cap\mathbb R^n_+\subset\mathbb R^n_+\setminus B^+\), and therefore \(D^\beta\widetilde v=0\) almost everywhere in \(B_\rho(x_0)\cap\mathbb R^n_+\). If \(|x_0'|=R\), take \(\rho:=\frac{\delta}{4}\). If \(x\in B_\rho(x_0)\cap\mathbb R^n_+\), then either \(x\notin B^+\), where \(D^\beta\widetilde v=0\), or \(x\in B^+\). In the latter case, \(|x|\geq R-\rho=R-\frac{\delta}{4}>R-\frac{\delta}{2}\), so \(x\in U\), and \(D^\beta v=0\) almost everywhere there. Thus, in both cases, \(D^\beta\widetilde v\) vanishes almost everywhere in a relative neighborhood of \(x_0\) in the half-space.

Since \(|\beta|\leq m-1\), we have \(D^\beta\widetilde v\in W^{1,p}(\mathbb R^n_+)\). Applying Lemma \ref{lem:locality_trace_half_space_order_m} to \(F=D^\beta\widetilde v\), we get \(\gamma_0D^\beta\widetilde v=0\) near \(x_0\). Since \(x_0\) was arbitrary, the trace vanishes on \(\partial\mathbb R^n_+\setminus B'\), and therefore on all of \(\partial\mathbb R^n_+\).

By Theorem \ref{thm:trace_half_space_order_m}, \(\widetilde v\in W^{m,p}_0(\mathbb R^n_+)\). Thus there exists \((\eta_k)_{k\in\mathbb N}\subset C_c^\infty(\mathbb R^n_+)\) such that \(\eta_k\to\widetilde v\) in \(W^{m,p}(\mathbb R^n_+)\). Choose \(\chi\in C_c^\infty(\mathbb R^n)\) such that \(\operatorname{supp}\chi\subset B_R(0)\) and \(\chi\equiv1\) on a neighborhood of \(K\). Since multiplication by \(\chi|_{\mathbb R^n_+}\) is continuous in \(W^{m,p}(\mathbb R^n_+)\), we get \(\chi\eta_k\to\chi\widetilde v=\widetilde v\) in \(W^{m,p}(\mathbb R^n_+)\). Moreover, \((\chi\eta_k)|_{B^+}\in C_c^\infty(B^+)\). Restricting to \(B^+\), we obtain a sequence in \(C_c^\infty(B^+)\) converging to \(v\) in \(W^{m,p}(B^+)\). Hence \(v\) belongs to the closure of \(C_c^\infty(B^+)\) in \(W^{m,p}(B^+)\).
\end{proof}
We shall use the following multiplication result repeatedly.
\begin{lemma}\label{lem:smooth_multiplication_sobolev_boundary}
Let \((M,g)\) be a smooth Riemannian manifold, with or without boundary, and  \(\textbf{E}\to M\) a smooth vector bundle equipped with a fiber metric and a compatible connection. Also, consider \(m\in\mathbb N\), \(1\leq p<\infty\), and suppose that \(\eta\in C^\infty(M)\) satisfies \(\|\nabla^j\eta\|_{L^\infty(M)}<\infty\) for every \(0\leq j\leq m\). Then \(u\mapsto\eta u\) is linear and continuous from \(W^{m,p}(\textbf{E})\) to itself. In particular, there exists \(C_{\eta,m}>0\) such that \(\|\eta u\|_{W^{m,p}(\textbf{E})}\leq C_{\eta,m}\|u\|_{W^{m,p}(\textbf{E})}\) for every \(u\in W^{m,p}(\textbf{E})\).

\end{lemma}
\begin{proof}
First let \(u\in W^{m,p}(\textbf{E})\cap\Gamma(\textbf{E})\). By the iterated Leibniz rule for the induced connection, for every \(0\leq s\leq m\) there exists \(A_s>0\), depending only on \(s\), such that
\[
|\nabla^s(\eta u)|_{g\otimes h_\textbf{E}}
\leq
A_s\sum_{t=0}^s
|\nabla^{s-t}\eta|_g\,|\nabla^t u|_{g\otimes h_\textbf{E}}.
\]
Taking \(L^p\)-norms and using the boundedness of the derivatives of \(\eta\), we get \[\|\nabla^s(\eta u)\|_{L^p}\leq A_s\displaystyle\sum_{t=0}^s\|\nabla^{s-t}\eta\|_{L^\infty}\|\nabla^t u\|_{L^p}.\] Summing over \(s=0,\dots,m\), we obtain the desired estimate for smooth sections.

For arbitrary \(u\in W^{m,p}(\textbf{E})\), we use \cite[Theorem 8]{VelazquezSandoval} to take \(u_k\in W^{m,p}(\textbf{E})\cap\Gamma(\textbf{E})\) with \(u_k\to u\) in \(W^{m,p}(\textbf{E})\). The estimate already proved shows that \((\eta u_k)_k\) is Cauchy in \(W^{m,p}(\textbf{E})\), and its limit is independent of the approximating sequence. This defines \(\eta u\), and the estimate follows by passing to the limit.
\end{proof}
We now globalize the trace construction. Compactness is used here in order to work with a finite boundary cover and obtain a global estimate without imposing additional uniformity assumptions on the geometry.

\begin{theorem}[Integer-order trace theorem for vector bundles]\label{thm:integer_order_trace_vector_bundle}
Let \((M,g)\) be a compact Riemannian manifold with smooth boundary, and \(\textbf{E}\to M\) be a smooth vector bundle of finite rank equipped with a fiber metric \(h_\textbf{E}\) and a compatible connection \(\nabla^\textbf{E}\). Also consider \(m\geq1\) and \(1\leq p<\infty\). Then the restriction map \(\gamma_0:\Gamma(\textbf{E})\to\Gamma(\textbf{E}|_{\partial M})\), \(\gamma_0(u)=u|_{\partial M}\), admits a unique continuous linear extension
\[
\operatorname{Tr}:W^{m,p}(\textbf{E})
\longrightarrow
W^{m-1,p}(\textbf{E}|_{\partial M}).
\]
In particular, there exists \(C>0\) such that \(\|\operatorname{Tr}(u)\|_{W^{m-1,p}(\textbf{E}|_{\partial M})}\leq C\|u\|_{W^{m,p}(\textbf{E})}\) for every \(u\in W^{m,p}(\textbf{E})\).
\end{theorem}

\begin{proof}
We first prove the estimate for smooth sections. Since \(\partial M\) is compact, take finitely many regular boundary charts \((U_\alpha,\phi_\alpha)_{\alpha=1}^N\) covering \(\partial M\). Write \(B_\alpha^+:=\phi_\alpha(U_\alpha\cap\operatorname{Int}(M))\) and \(B_\alpha':=\phi_\alpha(U_\alpha\cap\partial M)\). Refining the cover if necessary, assume that \(\textbf{E}\) is trivialized over each \(U_\alpha\) by a smooth local frame \(e_{1,\alpha},\dots,e_{r,\alpha}\). Let \((\psi_\alpha)_{\alpha=1}^N\) be a smooth partition of unity, defined on a neighborhood of \(\partial M\), subordinate to the \(U_\alpha\)'s, with \(\operatorname{supp}\psi_\alpha\Subset U_\alpha\), and such that \(\displaystyle\sum_{\alpha=1}^N\psi_\alpha=1\) near \(\partial M\).

Let \(u\in\Gamma(\textbf{E})\). On \(U_\alpha\), write \(u=\displaystyle\sum_{a=1}^r u_\alpha^a e_{a,\alpha}\). Then \(\psi_\alpha u=\displaystyle\sum_{a=1}^r\psi_\alpha u_\alpha^a e_{a,\alpha}\). For each component, define \(f_\alpha^a:=(\psi_\alpha u_\alpha^a)\circ\phi_\alpha^{-1}\) on \(B_\alpha^+\). Since \(\operatorname{supp}\psi_\alpha\Subset U_\alpha\), the function \(f_\alpha^a\) vanishes in a neighborhood of the spherical part of \(B_\alpha^+\). Let \(\widetilde f_\alpha^a\) be its extension by zero to \(\mathbb R^n_+\). By Lemma \ref{lem:zero_extension_half_ball_order_m}, \(\widetilde f_\alpha^a\in W^{m,p}(\mathbb R^n_+)\), and \(\|\widetilde f_\alpha^a\|_{W^{m,p}(\mathbb R^n_+)}=\|f_\alpha^a\|_{W^{m,p}(B_\alpha^+)}\). Applying Theorem \ref{thm:trace_half_space_order_m} with \(\beta=0\), and then restricting from \(\mathbb R^{n-1}\) to \(B_\alpha'\), we obtain
\[
\left\|
\left.(\gamma_0\widetilde f_\alpha^a)\right|_{B_\alpha'}
\right\|_{W^{m-1,p}(B_\alpha')}
\leq
C_\alpha
\|f_\alpha^a\|_{W^{m,p}(B_\alpha^+)}.
\]
For smooth \(u\), the function \(\left.(\gamma_0\widetilde f_\alpha^a)\right|_{B_\alpha'}\) coincides with \((f_\alpha^a)|_{B_\alpha'}\), because \(\widetilde f_\alpha^a=f_\alpha^a\) near the flat part of the boundary. Summing over \(a\) and using \cite[Lemma 5]{VelazquezSandoval}, we obtain
\[
\sum_{a=1}^r
\left\|
\left.(\gamma_0\widetilde f_\alpha^a)\right|_{B_\alpha'}
\right\|_{W^{m-1,p}(B_\alpha')}
\leq
C_\alpha'
\|\psi_\alpha u\|_{W^{m,p}(\textbf{E}|_{U_\alpha})}.
\]
On the other hand, applying the usual local norm equivalence on \(\partial M\) to the restricted bundle \(\textbf{E}|_{\partial M}\to\partial M\), we have
\[
\|(\psi_\alpha u)|_{\partial M}\|_{W^{m-1,p}(\textbf{E}|_{U_\alpha\cap\partial M})}
\leq
C_\alpha''
\sum_{a=1}^r
\left\|
\left.(\gamma_0\widetilde f_\alpha^a)\right|_{B_\alpha'}
\right\|_{W^{m-1,p}(B_\alpha')}.
\]
Therefore \(\|(\psi_\alpha u)|_{\partial M}\|_{W^{m-1,p}(\textbf{E}|_{U_\alpha\cap\partial M})}\leq C_\alpha'''\|\psi_\alpha u\|_{W^{m,p}(\textbf{E}|_{U_\alpha})}\). By Lemma \ref{lem:smooth_multiplication_sobolev_boundary}, applied to the fixed function \(\psi_\alpha\), there exists \(C_\alpha''''>0\) such that \(\|\psi_\alpha u\|_{W^{m,p}(\textbf{E})}\leq C_\alpha''''\|u\|_{W^{m,p}(\textbf{E})}\). Thus
\[
\|(\psi_\alpha u)|_{\partial M}\|_{W^{m-1,p}(\textbf{E}|_{U_\alpha\cap\partial M})}
\leq
C_\alpha'''''\|u\|_{W^{m,p}(\textbf{E})}.
\]

Since the cover is finite and \(\displaystyle\sum_{\alpha=1}^N\psi_\alpha=1\) near \(\partial M\), we have \(u|_{\partial M}=\displaystyle\sum_{\alpha=1}^N(\psi_\alpha u)|_{\partial M}\). Using the triangle inequality and the local equivalence of norms on \(\textbf{E}|_{\partial M}\), we obtain \(C>0\) such that
\[
\|u|_{\partial M}\|_{W^{m-1,p}(\textbf{E}|_{\partial M})}
\leq
C\|u\|_{W^{m,p}(\textbf{E})}
\]
for every \(u\in\Gamma(\textbf{E})\cap W^{m,p}(\textbf{E})\).

Now let \(u\in W^{m,p}(\textbf{E})\). By Meyers--Serrin theorem for vector bundles \cite[Theorem 8]{VelazquezSandoval}, there exists \((u_k)_{k\in\mathbb N}\subset W^{m,p}(\textbf{E})\cap \Gamma(\textbf{E})\) such that \(u_k\to u\) in \(W^{m,p}(\textbf{E})\). The estimate just proved gives $\|u_k|_{\partial M}-u_\ell|_{\partial M}\|_{W^{m-1,p}(\textbf{E}|_{\partial M})}
\leq
C\|u_k-u_\ell\|_{W^{m,p}(\textbf{E})}$.

Thus \((u_k|_{\partial M})_k\) is Cauchy in \(W^{m-1,p}(\textbf{E}|_{\partial M})\). We define \(\operatorname{Tr}(u):=\lim_{k\to\infty}u_k|_{\partial M}\) in \(W^{m-1,p}(\textbf{E}|_{\partial M})\). The same estimate shows that this definition does not depend on the approximating sequence. Passing to the limit in \(\|u_k|_{\partial M}\|_{W^{m-1,p}(\textbf{E}|_{\partial M})}\leq C\|u_k\|_{W^{m,p}(\textbf{E})}\), we get \(\|\operatorname{Tr}(u)\|_{W^{m-1,p}(\textbf{E}|_{\partial M})}\leq C\|u\|_{W^{m,p}(\textbf{E})}\). Uniqueness follows from the density of \(W^{m,p}(\textbf{E})\cap \Gamma(\textbf{E})\) in \(W^{m,p}(\textbf{E})\) \cite[Theorem 8]{VelazquezSandoval}.
\end{proof}

We now introduce traces of higher covariant derivatives. For \(j\geq0\), set \(F_j:=T^{(0,j)}(TM)\otimes \textbf{E}\), with the convention \(F_0=\textbf{E}\). The bundle \(F_j\) is equipped with the product metric and the connection induced by the Levi--Civita connection and \(\nabla^\textbf{E}\).

\begin{theorem}[Higher order traces]\label{thm:higher_order_traces_bundle}
Let \((M,g)\) be a compact Riemannian manifold with a smooth boundary, and \(\textbf{E}\to M\) be a smooth vector bundle of finite rank equipped with a fiber metric and a compatible connection. Also consider \(m\geq1\) and \(1\leq p<\infty\). For each \(0\leq j\leq m-1\), the smooth restriction operator \(u\mapsto \gamma_0(\nabla^j u)=(\nabla^j u)|_{\partial M}\), initially defined for \(u\in\Gamma(\textbf{E})\), extends uniquely to a continuous linear map
\[
\operatorname{Tr}_j:
W^{m,p}(\textbf{E})
\longrightarrow
W^{m-j-1,p}
\left(
\left(T^{(0,j)}(TM)\otimes \textbf{E}\right)\big|_{\partial M}
\right).
\]
Consequently, the operator
\[
\mathcal T_m:
W^{m,p}(\textbf{E})
\longrightarrow
\bigoplus_{j=0}^{m-1}
W^{m-j-1,p}
\left(
\left(T^{(0,j)}(TM)\otimes \textbf{E}\right)\big|_{\partial M}
\right),
\]
defined by \(\mathcal T_m u:=(\operatorname{Tr}_0u,\operatorname{Tr}_1u,\dots,\operatorname{Tr}_{m-1}u)\), is linear and continuous.
\end{theorem}

\begin{proof}
Fix \(0\leq j\leq m-1\). Let \(F_j=T^{(0,j)}(TM)\otimes \textbf{E}\). If \(u\in W^{m,p}(\textbf{E})\), then \(\nabla_w^j u\in W^{m-j,p}(F_j)\). Indeed, for \(0\leq s\leq m-j\), the weak derivative \(\nabla_w^s(\nabla_w^j u)\) is naturally identified with \(\nabla_w^{s+j}u\), as a section of \(T^{(0,s)}(TM)\otimes F_j\simeq T^{(0,s+j)}(TM)\otimes \textbf{E}\). Therefore \(\|\nabla_w^j u\|_{W^{m-j,p}(F_j)}\leq \|u\|_{W^{m,p}(\textbf{E})}\).

Since \(m-j\geq1\), we may apply Theorem \ref{thm:integer_order_trace_vector_bundle} to \(F_j\to M\). This gives a continuous linear map
\[
\operatorname{Tr}^{F_j}:
W^{m-j,p}(F_j)
\longrightarrow
W^{m-j-1,p}(F_j|_{\partial M}).
\]
Define \(\operatorname{Tr}_j u:=\operatorname{Tr}^{F_j}(\nabla_w^j u)\). Then
\[
\|\operatorname{Tr}_j u\|_{W^{m-j-1,p}(F_j|_{\partial M})}
\leq
C\|\nabla_w^j u\|_{W^{m-j,p}(F_j)}
\leq
C\|u\|_{W^{m,p}(\textbf{E})}.
\]
Thus \(\operatorname{Tr}_j\) is linear and continuous. If \(u\in\Gamma(\textbf{E})\), then \(\nabla_w^j u=\nabla^j u\), and by definition of the trace on \(F_j\), \(\operatorname{Tr}_j u=\operatorname{Tr}^{F_j}(\nabla^j u)=(\nabla^j u)|_{\partial M}\). Hence \(\operatorname{Tr}_j\) extends the smooth restriction operator. Uniqueness follows from the density of \(W^{m,p}(\textbf{E})\cap \Gamma(\textbf{E})\) in \(W^{m,p}(\textbf{E})\). Finally, since each \(\operatorname{Tr}_j\) is continuous, the product map \(\mathcal T_m\) is continuous for any norm on the finite direct sum.
\end{proof}

The Green formula for iterated covariant derivatives explains why \(\mathcal T_m\) is the natural boundary operator. Indeed, the boundary terms in integration by parts involve restrictions of the form \((\nabla^j u)|_{\partial M}\), with \(0\leq j\leq m-1\). Thus, for higher order problems, the vanishing of the trace of \(u\) alone is not the natural zero boundary condition; the traces of all covariant derivatives up to order \(m-1\) must vanish.

\begin{definition}[Zero trace Sobolev space]\label{def:zero_trace_space_order_m_bundle}
For \(m\in\mathbb N\) and \(1\leq p<\infty\), we define
\[
W^{m,p}_0(\textbf{E})
:=
\overline{\Gamma_{c,\operatorname{Int}(M)}(\textbf{E})}^{\,\|\cdot\|_{W^{m,p}(\textbf{E})}}.
\]
Thus \(W^{m,p}_0(\textbf{E})\) is the closure, in \(W^{m,p}(\textbf{E})\), of smooth sections whose compact support is contained in the interior of \(M\).
\end{definition}

\begin{theorem}[Characterization of \(W^{m,p}_0(\textbf{E})\) by traces]\label{thm:zero_trace_characterization_order_m_bundle}
Let \((M,g)\) be a compact Riemannian manifold with smooth boundary, and \(\textbf{E}\to M\) be a smooth vector bundle of finite rank equipped with a fiber metric and a compatible connection. Also consider \(m\geq1\) and \(1\leq p<\infty\). Then \(W^{m,p}_0(\textbf{E})=\ker(\mathcal T_m)\). \textbf{E}quivalently,
\[
W^{m,p}_0(\textbf{E})
=
\left\{
u\in W^{m,p}(\textbf{E})
\mathrel{\Big|}
\operatorname{Tr}_j u=0
\text{ for every }0\leq j\leq m-1
\right\}.
\]
\end{theorem}

\begin{proof}
We first prove that \(W^{m,p}_0(\textbf{E})\subseteq\ker(\mathcal T_m)\). Let \(u\in W^{m,p}_0(\textbf{E})\). By definition, there exists \(u_k\in\Gamma_{c,\operatorname{Int}(M)}(\textbf{E})\) such that \(u_k\to u\) in \(W^{m,p}(\textbf{E})\). Since each \(u_k\) has compact support contained in \(\operatorname{Int}(M)\), there is a neighborhood of \(\partial M\) where \(u_k\) vanishes. Hence \(\nabla^j u_k\) vanishes in a neighborhood of \(\partial M\) for every \(0\leq j\leq m-1\), and therefore \(\operatorname{Tr}_j u_k=0\). By continuity of \(\operatorname{Tr}_j\), \(\operatorname{Tr}_j u=\displaystyle\lim_{k\to\infty}\operatorname{Tr}_j u_k=0\). Since this holds for every \(j=0,\dots,m-1\), we get \(u\in\ker(\mathcal T_m)\).

Conversely, let \(u\in\ker(\mathcal T_m)\). Take a finite open cover \((U_\alpha)_{\alpha=1}^N\) of \(M\), formed by regular coordinate balls contained in \(\operatorname{Int}(M)\) and regular boundary charts near \(\partial M\). Let \((\psi_\alpha)_{\alpha=1}^N\) be a smooth partition of unity subordinate to this cover, with \(\operatorname{supp}\psi_\alpha\Subset U_\alpha\). Since \(u=\displaystyle\sum_{\alpha=1}^N\psi_\alpha u\), it is enough to prove that \(\psi_\alpha u\in W^{m,p}_0(\textbf{E})\) for each \(\alpha\).

We first prove that \(\operatorname{Tr}_j(\psi_\alpha u)=0\) for every \(0\leq j\leq m-1\). Since \(u\in\ker(\mathcal T_m)\), we have \(\operatorname{Tr}_t u=0\) for every \(0\leq t\leq m-1\). By \cite[Theorem 8]{VelazquezSandoval}, choose \(v_k\in W^{m,p}(\textbf{E})\cap \Gamma(\textbf{E})\) such that \(v_k\to u\) in \(W^{m,p}(\textbf{E})\). By continuity of the higher order traces, \(\operatorname{Tr}_t v_k\to\operatorname{Tr}_t u=0\) in \(W^{m-t-1,p}((T^{(0,t)}(TM)\otimes \textbf{E})|_{\partial M})\) for every \(0\leq t\leq m-1\).

Fix \(j\in\{0,\dots,m-1\}\). For smooth sections, the iterated Leibniz rule gives, up to the natural permutations of covariant indices,
\[
\nabla^j(\psi_\alpha v_k)
=
\sum_{t=0}^j
\binom jt
(\nabla^{j-t}\psi_\alpha)\otimes\nabla^t v_k.
\]
Restricting to \(\partial M\), we get
\[
\operatorname{Tr}_j(\psi_\alpha v_k)
=
\sum_{t=0}^j
\binom jt
\bigl((\nabla^{j-t}\psi_\alpha)|_{\partial M}\bigr)\otimes\operatorname{Tr}_t v_k,
\]
again up to the same natural permutations. These permutations are pointwise isometries, and multiplication by the fixed smooth tensors \((\nabla^{j-t}\psi_\alpha)|_{\partial M}\) is continuous on the corresponding Sobolev spaces over the compact boundary \(\partial M\). Since \(\operatorname{Tr}_t v_k\to0\) for every \(0\leq t\leq j\), it follows that \(\operatorname{Tr}_j(\psi_\alpha v_k)\to0\). On the other hand, \(v\mapsto\operatorname{Tr}_j(\psi_\alpha v)\) is continuous from \(W^{m,p}(\textbf{E})\) into the corresponding boundary Sobolev space, because multiplication by \(\psi_\alpha\) is continuous in \(W^{m,p}(\textbf{E})\) and \(\operatorname{Tr}_j\) is continuous. Since \(v_k\to u\) in \(W^{m,p}(\textbf{E})\), we also have \(\operatorname{Tr}_j(\psi_\alpha v_k)\to\operatorname{Tr}_j(\psi_\alpha u)\). By uniqueness of limits, \(\operatorname{Tr}_j(\psi_\alpha u)=0\). Since \(j\) was arbitrary, all higher order traces of \(\psi_\alpha u\) vanish.

We now distinguish the two types of charts. If \(U_\alpha\Subset\operatorname{Int}(M)\), then \(\psi_\alpha u\) is supported away from \(\partial M\). Choose \(\chi_\alpha\in C_c^\infty(\operatorname{Int}(M))\) such that \(\chi_\alpha\equiv1\) on a neighborhood of \(\operatorname{supp}\psi_\alpha\). By the Meyers--Serrin theorem for vector bundles \cite[Theorem 8]{VelazquezSandoval}, take \(z_k\in W^{m,p}(\textbf{E})\cap \Gamma(\textbf{E})\) such that \(z_k\to\psi_\alpha u\) in \(W^{m,p}(\textbf{E})\). By Lemma \ref{lem:smooth_multiplication_sobolev_boundary}, \(\chi_\alpha z_k\to\chi_\alpha\psi_\alpha u=\psi_\alpha u\) in \(W^{m,p}(\textbf{E})\). Moreover, each \(\chi_\alpha z_k\) is smooth and compactly supported in \(\operatorname{Int}(M)\). Thus \(\psi_\alpha u\in W^{m,p}_0(\textbf{E})\).

Suppose now that \(U_\alpha\) is a boundary chart. Write \(B_\alpha^+:=\phi_\alpha(U_\alpha\cap\operatorname{Int}(M))\) and \(B_\alpha':=\phi_\alpha(U_\alpha\cap\partial M)\). Let \(e_{1,\alpha},\dots,e_{r,\alpha}\) be a smooth local frame of \(\textbf{E}\) over \(U_\alpha\), and write \(\psi_\alpha u=\displaystyle\sum_{a=1}^r w_\alpha^a e_{a,\alpha}\). Set \(q_\alpha^a:=w_\alpha^a\circ\phi_\alpha^{-1}\) on \(B_\alpha^+\). Since \(\operatorname{supp}\psi_\alpha\Subset U_\alpha\), each \(q_\alpha^a\) vanishes outside a compact subset of \(B_\alpha^+\cup B_\alpha'\) at positive distance from the spherical part of \(B_\alpha^+\). Hence the convention above applies.

We claim that \(\gamma_0D^\beta q_\alpha^a=0\) on \(B_\alpha'\) for every \(|\beta|\leq m-1\) and every \(a=1,\dots,r\). \textbf{E}quivalently, if \(\widetilde q_\alpha^a\) denotes the extension by zero of \(q_\alpha^a\) to \(\mathbb R^n_+\), then \(\left.(\gamma_0D^\beta\widetilde q_\alpha^a)\right|_{B_\alpha'}=0\).

We prove this by induction on \(\ell=|\beta|\). For \(\ell=0\), the identity \(\operatorname{Tr}_0(\psi_\alpha u)=0\) means that the trace of \(\psi_\alpha u\) vanishes as a section of \(\textbf{E}|_{\partial M}\). In the chosen frame and boundary chart, the trace was constructed by taking components, extending them by zero to the half-space, and applying the half-space trace. Therefore \(\left.(\gamma_0\widetilde q_\alpha^a)\right|_{B_\alpha'}=0\) for every \(a\).

Assume that the claim has been proved for all derivatives of order at most \(\ell-1\), with \(1\leq\ell\leq m-1\). Since \(\operatorname{Tr}_\ell(\psi_\alpha u)=0\), the components of the trace of \(\nabla^\ell(\psi_\alpha u)\) vanish on \(B_\alpha'\). In the chosen chart and frame, the components of \(\nabla^\ell(\psi_\alpha u)\) have the form
\[
(\nabla^\ell(\psi_\alpha u))^a_{i_1\dots i_\ell}
=
\partial_{i_\ell}\cdots\partial_{i_1}w_\alpha^a
+
\sum_{b=1}^r
\sum_{|\beta|\leq\ell-1}
(B_\beta)^a_{b\,i_1\dots i_\ell}\,\partial^\beta w_\alpha^b.
\]
After passing to coordinates, this identity expresses these components as derivatives of order \(\ell\) of \(q_\alpha^a\), plus lower order terms with smooth coefficients. The identity is first clear for smooth sections. For \(\psi_\alpha u\in W^{m,p}(\textbf{E})\), take a smooth sequence converging to \(\psi_\alpha u\) in \(W^{m,p}(\textbf{E})\), apply the identity to the sequence, and pass to the limit using the continuity of weak derivatives, multiplication by smooth coefficients, and the trace operators after extension by zero. Hence the same identity holds at the level of boundary traces.

By the induction hypothesis, all lower order boundary traces appearing in the second sum vanish. Multiplication by smooth coefficients preserves this vanishing, because \(\gamma_0(hF)=h|_{\partial\mathbb R^n_+}\gamma_0F\) is immediate for smooth functions and extends by continuity to the Sobolev spaces involved. Since the trace of \(\nabla^\ell(\psi_\alpha u)\) vanishes, we get \(\left.\gamma_0(D^\beta\widetilde q_\alpha^a)\right|_{B_\alpha'}=0\) for every \(|\beta|=\ell\) and every \(a\). This closes the induction.

Thus \(\gamma_0D^\beta q_\alpha^a=0\) on \(B_\alpha'\), in the sense of the convention above, for every \(|\beta|\leq m-1\). By Corollary \ref{cor:local_characterization_half_ball_order_m}, each \(q_\alpha^a\) belongs to the closure of \(C_c^\infty(B_\alpha^+)\) in \(W^{m,p}(B_\alpha^+)\). Hence, for each \(a\), there exists \(\varphi_k^a\in C_c^\infty(B_\alpha^+)\) such that \(\varphi_k^a\to q_\alpha^a\) in \(W^{m,p}(B_\alpha^+)\). Define, on \(U_\alpha\cap\operatorname{Int}(M)\), \(s_k:=\displaystyle\sum_{a=1}^r(\varphi_k^a\circ\phi_\alpha)e_{a,\alpha}\). Since each \(\varphi_k^a\) has compact support contained in \(B_\alpha^+\), each \(s_k\) extends by zero to a smooth section with compact support contained in \(\operatorname{Int}(M)\). By the local norm equivalence in \cite[Lemma 5]{VelazquezSandoval}, \(s_k\to\psi_\alpha u\) in \(W^{m,p}(\textbf{E})\). Therefore \(\psi_\alpha u\in W^{m,p}_0(\textbf{E})\).

We have proved that every localized piece \(\psi_\alpha u\) belongs to \(W^{m,p}_0(\textbf{E})\). Since \(u=\displaystyle\sum_{\alpha=1}^N\psi_\alpha u\), we conclude that \(u\in W^{m,p}_0(\textbf{E})\). Thus \(\ker(\mathcal T_m)\subseteq W^{m,p}_0(\textbf{E})\), and the proof is complete.
\end{proof}

We will use the following Sobolev embeddings, which follow from \cite[Theorems 10 and 11]{VelazquezSandoval}.

\begin{lemma}\label{Embedding}
Let \((M,g)\) be a compact \(n\)-dimensional Riemannian manifold with smooth boundary, let \(\textbf{E}\to M\) be a smooth vector bundle of finite rank endowed with a fiber metric and a compatible connection, and let \(1\leq q<n\). Then the following assertions hold:
\begin{itemize}
\item[$(i)$] If \(\displaystyle 1\leq \ell\leq q^*:=\frac{nq}{n-q}\), then \(W^{1,q}_0(\textbf{E})\hookrightarrow L^\ell(\textbf{E})\) continuously.
\item[$(ii)$] If \(1\leq \ell<q^*\), then \(W^{1,q}_0(\textbf{E})\hookrightarrow L^\ell(\textbf{E})\) compactly.
\end{itemize}
\end{lemma}

Since \((M,\lambda_g)\) has a finite measure, we also have the following elementary consequence.

\begin{remark}\label{embedding_boundary}
Let \((M,g)\) be a compact Riemannian manifold with a smooth boundary, and let \(1<p<q<\infty\). Then the embedding \(W^{1,q}_0(\textbf{E})\hookrightarrow W^{1,p}_0(\textbf{E})\) is continuous.
\end{remark}

We shall also use the compatibility between traces and fiber norms.

\begin{lemma}[Compatibility between traces and fiber norms]
\label{lem:trace_norm_compatibility}
Let $(M,g)$ be a compact Riemannian manifold with smooth boundary, and let
$\mathbf E\to M$ be a smooth vector bundle of finite rank equipped with a
fiber metric $h_{\mathbf E}$ and a compatible connection
$\nabla^{\mathbf E}$. If $u\in W^{1,q}(\mathbf E)$, with
$1\leq q<\infty$, then $|u|_{h_{\mathbf E}}\in W^{1,q}(M)$ and
\[
\operatorname{Tr}(|u|_{h_{\mathbf E}})
=
|\operatorname{Tr}_0u|_{h_{\mathbf E}}
\]
almost everywhere on $\partial M$.
\end{lemma}

\begin{proof}
By the weak Kato inequality, $|u|_{h_{\mathbf E}}\in W^{1,q}(M)$. Let
$(u_j)\subseteq\Gamma(\mathbf E)\cap W^{1,q}(\mathbf E)$ satisfy
$u_j\to u$ in $W^{1,q}(\mathbf E)$. The reverse triangle inequality
gives $|u_j|_{h_{\mathbf E}}\to|u|_{h_{\mathbf E}}$ in $L^q(M)$.
Applying the weak Kato inequality to each $u_j$, we also have
$|u_j|_{h_{\mathbf E}}\in W^{1,q}(M)$ and
\[
\left|\nabla_w|u_j|_{h_{\mathbf E}}\right|_g
\leq
|\nabla_w^{\mathbf E}u_j|_{g,h_{\mathbf E}}
\]
almost everywhere. Since
$\nabla_w^{\mathbf E}u_j\to\nabla_w^{\mathbf E}u$ in
$L^q(T^*M\otimes\mathbf E)$, the sequence
$\bigl(\nabla_w|u_j|_{h_{\mathbf E}}\bigr)$ is bounded in $L^q(T^*M)$.

If $1<q<\infty$, reflexivity provides a subsequence and
$Z\in L^q(T^*M)$ such that
\[
\nabla_w|u_j|_{h_{\mathbf E}}
\rightharpoonup Z
\qquad\text{in }L^q(T^*M).
\]
If $q=1$, the convergence
$|\nabla_w^{\mathbf E}u_j|_{g,h_{\mathbf E}}\to
|\nabla_w^{\mathbf E}u|_{g,h_{\mathbf E}}$ in $L^1(M)$ implies that
$\{|\nabla_w^{\mathbf E}u_j|_{g,h_{\mathbf E}}\mid j\in\mathbb N\}$ is
uniformly integrable. The weak Kato inequality then shows that
$\{\left|\nabla_w|u_j|_{h_{\mathbf E}}\right|_g\mid j\in\mathbb N\}$ is
uniformly integrable as well. Hence the Dunford--Pettis theorem for vector
bundles, applied to $T^*M$, together with the Eberlein--Šmulian theorem,
provides a subsequence and $Z\in L^1(T^*M)$ such that
\[
\nabla_w|u_j|_{h_{\mathbf E}}
\rightharpoonup Z
\qquad\text{in }L^1(T^*M).
\]

Thus, in either case, after passing to a subsequence,
$\nabla_w|u_j|_{h_{\mathbf E}}\rightharpoonup Z$ in $L^q(T^*M)$. For
every $\eta\in\Gamma_{c,\operatorname{Int}(M)}(T^*M)$, the definition of
the weak covariant derivative gives
\[
\int_M
\left\langle
\nabla_w|u_j|_{h_{\mathbf E}},\eta
\right\rangle_g\,d\lambda_g
=
\int_M
|u_j|_{h_{\mathbf E}}\nabla^*\eta\,d\lambda_g.
\]
Passing to the limit, using the weak convergence on the left and the
strong convergence $|u_j|_{h_{\mathbf E}}\to|u|_{h_{\mathbf E}}$ in
$L^q(M)$ on the right, gives
\[
\int_M\langle Z,\eta\rangle_g\,d\lambda_g
=
\int_M|u|_{h_{\mathbf E}}\nabla^*\eta\,d\lambda_g.
\]
Hence
\[
Z=\nabla_w|u|_{h_{\mathbf E}}.
\]
Since $|u_j|_{h_{\mathbf E}}\to|u|_{h_{\mathbf E}}$ strongly in
$L^q(M)$ and
$\nabla_w|u_j|_{h_{\mathbf E}}\rightharpoonup
\nabla_w|u|_{h_{\mathbf E}}$ in $L^q(T^*M)$, the canonical isometric
embedding of $W^{1,q}(M)$ into
$L^q(M)\times L^q(T^*M)$ yields
\[
|u_j|_{h_{\mathbf E}}
\rightharpoonup
|u|_{h_{\mathbf E}}
\qquad\text{in }W^{1,q}(M).
\]

By Theorem~\ref{thm:integer_order_trace_vector_bundle}, applied to the
trivial line bundle and to $\mathbf E$, respectively, the operators
$\operatorname{Tr}:W^{1,q}(M)\to L^q(\partial M)$ and
$\operatorname{Tr}_0:W^{1,q}(\mathbf E)\to
L^q(\mathbf E|_{\partial M})$ are linear and continuous. Therefore
\[
\operatorname{Tr}(|u_j|_{h_{\mathbf E}})
\rightharpoonup
\operatorname{Tr}(|u|_{h_{\mathbf E}})
\qquad\text{in }L^q(\partial M),
\]
while
$\operatorname{Tr}_0u_j\to\operatorname{Tr}_0u$ in
$L^q(\mathbf E|_{\partial M})$. The reverse triangle inequality then
implies
\[
|\operatorname{Tr}_0u_j|_{h_{\mathbf E}}
\longrightarrow
|\operatorname{Tr}_0u|_{h_{\mathbf E}}
\qquad\text{in }L^q(\partial M).
\]
Since $u_j$ is smooth, the trace agrees with the restriction to the
boundary and therefore
\[
\operatorname{Tr}(|u_j|_{h_{\mathbf E}})
=
|\operatorname{Tr}_0u_j|_{h_{\mathbf E}}
\]
almost everywhere on $\partial M$ for every $j$. The latter strong
convergence also implies weak convergence, so the same sequence converges
weakly both to $\operatorname{Tr}(|u|_{h_{\mathbf E}})$ and to
$|\operatorname{Tr}_0u|_{h_{\mathbf E}}$. By uniqueness of weak limits,
\[
\operatorname{Tr}(|u|_{h_{\mathbf E}})
=
|\operatorname{Tr}_0u|_{h_{\mathbf E}}
\]
almost everywhere on $\partial M$.
\end{proof}

By using the trace characterization of \(W^{1,q}_0(\textbf{E})\), we obtain the following Poincaré inequality.

\begin{theorem}[Poincaré inequality on \(W^{1,q}_0(\textbf{E})\)]\label{teo:poincare}
Let \((M,g)\) be a compact Riemannian manifold with smooth boundary, and assume that every connected component of \(M\) has nonempty boundary. Let \(\textbf{E}\to M\) be a smooth vector bundle of finite rank endowed with a fiber metric \(h_\textbf{E}\) and a compatible connection \(\nabla^\textbf{E}\). If \(1\leq q<n\), then there exists a constant \(C_P>0\), depending only on \(M,g,\textbf{E},h_\textbf{E},\nabla^\textbf{E}\) and \(q\), such that \(\displaystyle \|u\|_{L^q(\textbf{E})}\leq C_P\|\nabla^\textbf{E} u\|_{L^q(T^*M\otimes \textbf{E})}\) for every \(u\in W^{1,q}_0(\textbf{E})\). Consequently, \(\displaystyle \|u\|:=\left(\int_M|\nabla^\textbf{E} u|_{g\otimes h_\textbf{E}}^q\,d\lambda_g\right)^{1/q}\) defines a norm on \(W^{1,q}_0(\textbf{E})\) equivalent to the usual Sobolev norm.
\end{theorem}

\begin{proof}
We argue by contradiction. Suppose that the estimate is false. Then there exists a sequence \((u_k)\subset W^{1,q}_0(\textbf{E})\) such that \(\displaystyle \|u_k\|_{L^q(\textbf{E})}=1\) and \(\displaystyle \|\nabla^\textbf{E} u_k\|_{L^q(T^*M\otimes \textbf{E})}\to0\). In particular, \((u_k)\) is bounded in \(W^{1,q}_0(\textbf{E})\). Since \(q<n\), we have \(\displaystyle q<q^*=\frac{nq}{n-q}\), and Lemma~\ref{Embedding} gives the compact embedding \(W^{1,q}_0(\textbf{E})\hookrightarrow L^q(\textbf{E})\). Hence, after passing to a subsequence, there exists \(u\in L^q(\textbf{E})\) such that \(u_k\to u\) strongly in \(L^q(\textbf{E})\).

On the other hand, \(\nabla^\textbf{E} u_k\to0\) strongly in \(L^q(T^*M\otimes \textbf{E})\). Since the weak covariant derivative is closed as an operator from \(L^q(\textbf{E})\) into \(L^q(T^*M\otimes \textbf{E})\), it follows that \(u\in W^{1,q}(\textbf{E})\) and \(\nabla^\textbf{E} u=0\). Therefore \(u_k\to u\) strongly in \(W^{1,q}(\textbf{E})\). Since \(W^{1,q}_0(\textbf{E})\) is closed in \(W^{1,q}(\textbf{E})\), we also have \(u\in W^{1,q}_0(\textbf{E})\). By Theorem~\ref{thm:zero_trace_characterization_order_m_bundle}, applied with \(m=1\), this means that \(\operatorname{Tr}_0u=0\) on \(\partial M\).

It remains to prove that \(u=0\). By the weak Kato inequality for metric connections, \(|u|_{h_\textbf{E}}\in W^{1,q}(M)\) and \(\displaystyle |d|u|_{h_\textbf{E}}|_g\leq|\nabla^\textbf{E} u|_{g\otimes h_\textbf{E}}\) almost everywhere. Since \(\nabla^\textbf{E} u=0\), we get \(d|u|_{h_\textbf{E}}=0\) weakly. Hence \(|u|_{h_\textbf{E}}\) is constant almost everywhere on each connected component of \(M\). Moreover, by Lemma~\ref{lem:trace_norm_compatibility}, \(\operatorname{Tr}(|u|_{h_\textbf{E}})=|\operatorname{Tr}_0u|_{h_\textbf{E}}=0\) on \(\partial M\). Since every connected component of \(M\) meets the boundary, the constant value of \(|u|_{h_\textbf{E}}\) on each component must be zero. Thus, \(u=0\) almost everywhere on \(M\).

This contradicts the strong convergence \(u_k\to u\) in \(L^q(\textbf{E})\), because \(\displaystyle \|u\|_{L^q(\textbf{E})}=\lim_{k\to\infty}\|u_k\|_{L^q(\textbf{E})}=1\). Therefore, the desired constant \(C_P>0\) exists. The equivalence between \(\|\cdot\|\) and the usual \(W^{1,q}\)-norm on \(W^{1,q}_0(\textbf{E})\) follows immediately from the Poincaré inequality.
\end{proof}

\section{The variational framework and some technical lemmas}\label{sec:variational_framework}

Let \(X:=W^{1,q}_0(\textbf{E})\). In view of Theorem~\ref{teo:poincare}, we shall use on \(X\) the equivalent norm
\[
\|u\|_X
:=
\left(
\int_M|\nabla^\textbf{E} u|_{g\otimes h_\textbf{E}}^q\,d\lambda_g
\right)^{1/q}.
\]
In what follows, for \(u\in X\), the symbol \(\nabla^\textbf{E} u\) denotes its weak covariant derivative.

\begin{definition}
A section \(u\in X=W^{1,p}_0(\mathbf E)\) is said to be a
\emph{weak solution} of \eqref{ProblemP} if
\[
\int_M
a\!\left(|\nabla^{\mathbf E}u|_{g\otimes h_{\mathbf E}}^{p}\right)
|\nabla^{\mathbf E}u|_{g\otimes h_{\mathbf E}}^{p-2}
\left\langle
\nabla^{\mathbf E}u,\nabla^{\mathbf E}\varphi
\right\rangle_{g\otimes h_{\mathbf E}}
\,d\lambda_g
=
\int_M
\left\langle
f_x(u),\varphi
\right\rangle_{h_{\mathbf E}}
\,d\lambda_g,
\]
for every section \(\varphi\in W^{1,p}_0(\mathbf E)\).
\end{definition}

\begin{lemma}\label{lem:fiberwise_growth_estimates}
Assume that \ref{Bi2f_1} and \ref{Bi2f_2} hold. Then, for every \(\varepsilon>0\), there exists \(C_\varepsilon>0\) such that
\begin{equation}\label{estimative_bundle_f}
|f_x(\xi)|_{h_\textbf{E}}
\leq
\varepsilon |\xi|_{h_\textbf{E}}^{p-1}
+
C_\varepsilon |\xi|_{h_\textbf{E}}^{r-1}
\end{equation}
for every \(x\in M\) and every \(\xi\in \textbf{E}_x\). Moreover, 
\begin{equation}\label{estimative_bundle_F}
|\mathcal F_x(\xi)|
\leq
\frac{\varepsilon}{p}|\xi|_{h_\textbf{E}}^{p}
+
\frac{C_\varepsilon}{r}|\xi|_{h_\textbf{E}}^{r}
\end{equation}
for every \(x\in M\) and every \(\xi\in \textbf{E}_x\).
\end{lemma}

\begin{proof}
Let \(\varepsilon>0\) be fixed. By \ref{Bi2f_1}, there exists \(\delta\in(0,1)\) such that $|f_x(\xi)|_{h_\textbf{E}}
\leq \varepsilon |\xi|_{h_\textbf{E}}^{p-1}$
whenever \(x\in M\) and \(0<|\xi|_{h_\textbf{E}}\leq\delta\). In particular, since each \(f_x\) is continuous in the fiber variable, letting \(\xi\to0_x\) gives \(f_x(0_x)=0\). Hence, the same estimate also holds for \(\xi=0_x\). On the other hand, if \(|\xi|_{h_\textbf{E}}\geq\delta\), then \ref{Bi2f_2} gives
\[
|f_x(\xi)|_{h_\textbf{E}}
\leq
C_f(1+|\xi|_{h_\textbf{E}}^{r-1})
\leq
C_f(1+\delta^{1-r})|\xi|_{h_\textbf{E}}^{r-1}.
\]
Taking \(C_\varepsilon\geq C_f(1+\delta^{1-r})\), we obtain \eqref{estimative_bundle_f} on the region \(|\xi|_{h_\textbf{E}}\geq\delta\), while the previous smallness estimate gives it on the region \(|\xi|_{h_\textbf{E}}\leq\delta\). This proves \eqref{estimative_bundle_f} for every \(x\in M\) and every \(\xi\in \textbf{E}_x\).

We now prove the estimate for the potential. Since \(\mathcal F_x(0_x)=0\) and \(f_x\) is the fiber gradient of \(\mathcal F_x\), for every \(\xi\in \textbf{E}_x\) one has $\mathcal F_x(\xi)=\int_0^1
\left\langle f_x(t\xi),\xi\right\rangle_{h_\textbf{E}}\,dt$.
Therefore, using \eqref{estimative_bundle_f},
\[
\begin{aligned}
|\mathcal F_x(\xi)|
\leq
\int_0^1
|f_x(t\xi)|_{h_\textbf{E}}\,|\xi|_{h_\textbf{E}}\,dt\leq
\int_0^1
\left(
\varepsilon t^{p-1}|\xi|_{h_\textbf{E}}^{p-1}
+
C_\varepsilon t^{r-1}|\xi|_{h_\textbf{E}}^{r-1}
\right)
|\xi|_{h_\textbf{E}}\,dt =
\frac{\varepsilon}{p}|\xi|_{h_\textbf{E}}^{p}
+
\frac{C_\varepsilon}{r}|\xi|_{h_\textbf{E}}^{r}.
\end{aligned}
\]
This proves \eqref{estimative_bundle_F}.
\end{proof}

\noindent The variational functional associated with problem \eqref{ProblemP} is
\begin{equation}\label{energy_functional}
I(u)
:=
\frac1p
\int_M
A\!\left(|\nabla^{\mathbf E}u|_{g\otimes h_{\mathbf E}}^{p}\right)
\,d\lambda_g
-
\int_M
\mathcal{F}_x(u)
\,d\lambda_g,
\qquad u\in X.
\end{equation}
It follows from assumption \ref{Bi2a_1} and Lemmas~\ref{Embedding} and
\ref{lem:fiberwise_growth_estimates} that \(I\) is well defined. Moreover, for every test section \(\varphi\in X\),
\[
\begin{aligned}
\langle I'(u),\varphi\rangle
&=
\int_M
a\!\left(|\nabla^{\mathbf E}u|_{g\otimes h_{\mathbf E}}^{p}\right)
|\nabla^{\mathbf E}u|_{g\otimes h_{\mathbf E}}^{p-2}
\left\langle
\nabla^{\mathbf E}u,
\nabla^{\mathbf E}\varphi
\right\rangle_{g\otimes h_{\mathbf E}}
\,d\lambda_g  
-
\int_M
\left\langle
f_x(u),
\varphi
\right\rangle_{h_{\mathbf E}}
\,d\lambda_g.
\end{aligned}
\]
In view of Proposition~\ref{prop:energy_functional_C1}, the functional
\(I:X\to\mathbb{R}\) is of class \(C^{1}\). Moreover, its critical points are exactly the weak solutions of \eqref{ProblemP}. The corresponding Nehari manifold is defined by
\[
\mathcal N
=
\left\{
u\in X\setminus\{0\}:
\langle I'(u),u\rangle=0
\right\}.
\]
Equivalently,  \(u\in\mathcal N\) if and only if \(u\neq0\) and
\begin{equation}\label{nehari_identity}
\int_M
a(|\nabla^\textbf{E} u|_{g\otimes h_\textbf{E}}^{p})
|\nabla^\textbf{E} u|_{g\otimes h_\textbf{E}}^{p}\,d\lambda_g
=
\int_M
\langle f(u),u\rangle_{h_\textbf{E}}\,d\lambda_g.
\end{equation}
We define the ground state energy level by 
\begin{equation}\label{Nehari_Level}
    c:=\displaystyle\inf_{u\in\mathcal N}I(u)
\end{equation}

\begin{lemma}\label{lem:lower_bound_F_bundle}
Assume that \ref{Bi2f_3} holds. Then there exist constants \(D>0\) and \(D'\geq0\) such that
\begin{equation}\label{lower_bound_F_bundle}
\mathcal F_x(\xi)\geq D|\xi|_{h_\textbf{E}}^{\theta}-D'
\end{equation}
for every \(x\in M\) and every \(\xi\in \textbf{E}_x\).
\end{lemma}

\begin{proof}
Let \(\theta\) be the exponent given by \ref{Bi2f_3}. Fix \(x\in M\) and \(\eta\in \textbf{E}_x\setminus\{0_x\}\). For \(t>0\), define \(\displaystyle H_{x,\eta}(t):=\frac{\mathcal F_x(t\eta)}{t^\theta}\). Since \(\mathcal F_x(t\eta)>0\), the function \(\log H_{x,\eta}\) is well defined. Moreover,
\[
\frac{d}{dt}\log H_{x,\eta}(t)
=
\frac{\langle f_x(t\eta),\eta\rangle_{h_\textbf{E}}}{\mathcal F_x(t\eta)}
-
\frac{\theta}{t}
=
\frac{\langle f_x(t\eta),t\eta\rangle_{h_\textbf{E}}-\theta\mathcal F_x(t\eta)}
{t\mathcal F_x(t\eta)}.
\]
By \ref{Bi2f_3}, the numerator is nonnegative, so \(H_{x,\eta}\) is nondecreasing on \((0,+\infty)\).

Consider the unit sphere bundle $S(\textbf{E}):=\{\eta\in \textbf{E}\mid |\eta|_{h_\textbf{E}}=1\}$. Since \(M\) is compact and \(\textbf{E}\) has finite rank, \(S(\textbf{E})\) is compact. The map \(\eta\mapsto\mathcal F_{\pi(\eta)}(\eta)\) is continuous on \(S(\textbf{E})\), and \(\mathcal F_{\pi(\eta)}(\eta)>0\) for every \(\eta\in S(\textbf{E})\) by \ref{Bi2f_3}. Hence $D:=\min_{\eta\in S(\textbf{E})}\mathcal F_{\pi(\eta)}(\eta)>0$. If \(|\xi|_{h_\textbf{E}}\ge1\), write \(\xi=\rho\eta\), where \(\rho:=|\xi|_{h_\textbf{E}}\ge1\) and \(|\eta|_{h_\textbf{E}}=1\). Since \(H_{x,\eta}\) is nondecreasing, 
\[
\mathcal F_x(\xi)=\mathcal F_x(\rho\eta)\ge\rho^\theta\mathcal F_x(\eta)\ge D|\xi|_{h_\textbf{E}}^{\theta}.
\]

It remains to consider the case \(|\xi|_{h_\textbf{E}}\le1\). Since the closed unit ball bundle $\{\xi\in \textbf{E}\mid |\xi|_{h_\textbf{E}}\le1\}$ is compact, the continuous function \(\xi\mapsto D|\xi|_{h_\textbf{E}}^{\theta}-\mathcal F_{\pi(\xi)}(\xi)\) attains its maximum. Define $D':=\max_{|\xi|_{h_\textbf{E}}\le1}
\Big(D|\xi|_{h_\textbf{E}}^{\theta}-\mathcal F_{\pi(\xi)}(\xi)\Big)$. Since \(0_x\) belongs to the closed unit ball bundle and \(\mathcal F_x(0_x)=0\), we have $D'\ge D|0_x|_{h_\textbf{E}}^{\theta}-\mathcal F_x(0_x)=0$. Therefore, $\mathcal F_x(\xi)\ge D|\xi|_{h_\textbf{E}}^{\theta}-D'$ for every \(\xi\in \textbf{E}_x\) with \(|\xi|_{h_\textbf{E}}\le1\). Combining this estimate with the previous one proves \eqref{lower_bound_F_bundle}.
\end{proof}

A nontrivial critical point \(u\in X\) of \(I\) such that \(I(u)=c\) will be called a ground state solution of \eqref{ProblemP}.

\noindent In the next result we prove that \(\mathcal N\) is nonempty and that the functional \(I\) is bounded from below on \(\mathcal N\).

\begin{lemma}\label{Nehari_projection}
For each $u\in X\setminus\{0\}$, there exists a unique $t_u>0$ such that
$t_u u\in\mathcal N$. Moreover, the map
$\gamma_u:(0,+\infty)\to\mathbb R$, defined by $\gamma_u(t):=I(tu)$, is
strictly increasing on $(0,t_u)$ and strictly decreasing on
$(t_u,+\infty)$. Consequently, $t_u$ is the unique global maximum point of
$\gamma_u$. In addition, $I(u)>0$ for every $u\in\mathcal N$.
\end{lemma}

\begin{proof}
Let \(u\in X\setminus\{0\}\) be fixed and consider the fibering map $\gamma_u:[0,\infty)\longrightarrow\mathbb{R}$, $\gamma_u(t):=I(tu)$. Since \(I\in C^{1}(X,\mathbb{R})\), the map \(\gamma_u\) is continuous on
\([0,\infty)\) and of class \(C^{1}\) on \((0,\infty)\). Moreover, the chain
rule yields
\[
\gamma_u'(t)=\langle I'(tu),u\rangle,
\qquad t>0.
\]
Therefore,
\[
tu\in\mathcal N
\quad\Longleftrightarrow\quad
\langle I'(tu),tu\rangle=0
\quad\Longleftrightarrow\quad
t\,\gamma_u'(t)=0.
\]
Since \(t>0\), this is equivalent to $\gamma_u'(t)=0$. We begin by proving the existence of such a parameter \(t_u\).
Combining \eqref{eq:A_lower_from_monotonicity} with assumption
\ref{Bi2a_1}, we obtain
\begin{equation*}
    A(\rho^{p})
\ge
\frac{p}{q}a(\rho^{p})\rho^{p}
\ge
\frac{pk_{1}}{q}\rho^{p}
+
\frac{pk_{2}}{q}\rho^{q},
\qquad
\rho\ge0.
\end{equation*}

Furthermore, by \eqref{estimative_bundle_F} in
Lemma~\ref{lem:fiberwise_growth_estimates}, we obtain
\begin{align}
\gamma_u(t)
&=
\frac1p
\int_M
A\!\left(t^p|\nabla^{\mathbf E}u|_{g\otimes h_{\mathbf E}}^{p}\right)
\,d\lambda_g
-
\int_M
\mathcal F_x(tu)
\,d\lambda_g
\nonumber\\
&\ge
\frac{k_1}{q}t^{p}
\int_M
|\nabla^{\mathbf E}u|_{g\otimes h_{\mathbf E}}^{p}
\,d\lambda_g
+
\frac{k_2}{q}t^{q}
\int_M
|\nabla^{\mathbf E}u|_{g\otimes h_{\mathbf E}}^{q}
\,d\lambda_g 
-
\frac{\varepsilon}{p}t^{p}
\int_M
|u|_{h_{\mathbf E}}^{p}
\,d\lambda_g
-
\frac{C_\varepsilon}{r}t^{r}
\int_M
|u|_{h_{\mathbf E}}^{r}
\,d\lambda_g.
\label{Projection1}
\end{align}
Using Theorem \ref{teo:poincare}, there exists a constant \(C_P>0\) such that
\[
\|u\|_{L^{p}(\mathbf E)}^{p}
\le
C_P
\|\nabla^{\mathbf E}u\|_{L^{p}(T^{*}M\otimes\mathbf E)}^{p}.
\]
Substituting this estimate into \eqref{Projection1} yields
\[
\gamma_u(t)
\ge
\left(
\frac{k_1}{q}
-
\frac{\varepsilon C_P}{p}
\right)
t^p
\int_M
|\nabla^{\mathbf E}u|_{g\otimes h_{\mathbf E}}^{p}
\,d\lambda_g
-
\frac{C_\varepsilon}{r}
t^r
\int_M
|u|_{h_{\mathbf E}}^{r}
\,d\lambda_g,
\]
where we have discarded the nonnegative term involving
\(t^q\). Choosing $\varepsilon \in \Big( 0,  \frac{pk_1}{qC_P}\Big)$. Therefore, defining
\[
C_1=
\left(
\frac{k_1}{q}
-
\frac{\varepsilon C_P}{p}
\right)
\int_M
|\nabla^{\mathbf E}u|_{g\otimes h_{\mathbf E}}^{p}
\,d\lambda_g>0,
\]
and
\[
C_2=
\frac{C_\varepsilon}{r}
\int_M
|u|_{h_{\mathbf E}}^{r}
\,d\lambda_g>0,
\]
we conclude that
\[
\gamma_u(t)
\ge
C_1t^p-C_2t^r,
\qquad t>0.
\]
Since \(r>p\), it follows that $\gamma_u(t)>0$ for all sufficiently small \(t>0\).

On the other hand, by Lemma~\ref{lem:lower_bound_F_bundle}, there exist
constants \(D>0\) and \(D'\ge0\) such that
\[
\mathcal F_x(\xi)
\ge
D|\xi|_{h_{\mathbf E}}^{\theta}-D',
\qquad
x\in M,\ \xi\in\mathbf E_x.
\]
Moreover, by assumption \ref{Bi2a_1}, there exists a constant \(C>0\) such
that
\[
A(\rho^p)\le C(\rho^p+\rho^q),
\qquad
\rho\ge0.
\]
Hence,
\begin{align}
\gamma_u(t)
&\leq
Ct^p
\int_M
|\nabla^{\mathbf E}u|_{g\otimes h_{\mathbf E}}^p
\,d\lambda_g
+
Ct^q
\int_M
|\nabla^{\mathbf E}u|_{g\otimes h_{\mathbf E}}^q
\,d\lambda_g-Dt^\theta
\int_M
|u|_{h_{\mathbf E}}^\theta
\,d\lambda_g
+
D'\lambda_g(M).
\label{eq:gamma_upper_bound_large_t}
\end{align}
Since \(u\neq0\), we have
\[
\int_M|u|_{h_{\mathbf E}}^\theta\,d\lambda_g>0.
\]
Because \(\theta>q\ge p\), the negative term of order \(t^\theta\)
dominates the positive terms of orders \(t^p\) and \(t^q\) as
\(t\to\infty\). Consequently,
\[
\gamma_u(t)\to-\infty
\qquad\text{as }t\to\infty.
\]
Since \(\gamma_u(t)>0\) for all sufficiently small \(t>0\), it follows that
\(\gamma_u\) attains a global maximum at some point \(t_u>0\). As
\(\gamma_u\in C^{1}(0,\infty)\), we have $\gamma_u'(t_u)=0$, which is equivalent to
\[
\langle I'(t_u u),t_u u\rangle=0.
\]
Hence, \(t_u u\in\mathcal N\), establishing the existence of the required
parameter \(t_u\).  This proves that $\mathcal N\neq\varnothing$.

We now establish the uniqueness of \(t_u\). Let \(t,s>0\) be such that
\(tu,su\in\mathcal N\). Since
\[
\langle I'(\tau u),\tau u\rangle=0,
\qquad \tau=t,s,
\]
dividing the corresponding Nehari identities by \(\tau^q\) yields
\begin{eqnarray}\label{Projection4}
 \int_M
a(t^p|\nabla^\textbf{E} u|_{g\otimes h_\textbf{E}}^p)
t^{p-q}|\nabla^\textbf{E} u|_{g\otimes h_\textbf{E}}^p\,d\lambda_g
=
\int_M
\frac{\langle f(tu),tu\rangle_{h_\textbf{E}}}{t^q}\,d\lambda_g   
\end{eqnarray}
and
\begin{equation}\label{Projection5}
  \int_M
a(s^p|\nabla^\textbf{E} u|_{g\otimes h_\textbf{E}}^p)
s^{p-q}|\nabla^\textbf{E} u|_{g\otimes h_\textbf{E}}^p\,d\lambda_g
=
\int_M
\frac{\langle f(su),su\rangle_{h_\textbf{E}}}{s^q}\,d\lambda_g.  
\end{equation}

Assume, by contradiction, that \(0<t<s\).
Since assumption~\ref{Bi2a_2} asserts that
\begin{equation*}
    \tau \to \frac{a(\tau^p)}{\tau^{q-p}}
\end{equation*}
is nonincreasing on \((0,\infty)\), we obtain
\begin{equation*}
    a(t^p|\nabla^\textbf{E} u|_{g\otimes h_\textbf{E}}^p)t^{p-q}
\geq
a(s^p|\nabla^\textbf{E} u|_{g\otimes h_\textbf{E}}^p)s^{p-q}.
\end{equation*}
whenever \(|\nabla^{\mathbf E}u|>0\). Hence,
\[
\int_M
\Bigl(
a(t^p|\nabla^{\mathbf E}u|^p)t^{p-q}
-
a(s^p|\nabla^{\mathbf E}u|^p)s^{p-q}
\Bigr)
|\nabla^{\mathbf E}u|^p\,d\lambda_g
\ge0.
\]
On the other hand, the set $\Omega_u=\{x\in M\mid u(x)\neq0\}$ has positive measure. By assumption~\ref{Bi2f_4}, for every
\(x\in\Omega_u\) the function
\[
\tau\to
\frac{\langle f_x(\tau u(x)),\tau u(x)\rangle_{h_\textbf{E}} }{\tau^q}
\]
is strictly increasing on \((0,\infty)\). Therefore,

\[
\frac{\langle f(tu),tu\rangle_{h_\textbf{E}}}{t^q}
-
\frac{\langle f(su),su\rangle_{h_\textbf{E}}}{s^q}
<0
\qquad
\text{a.e. on }\Omega_u.
\]
On $M\setminus\Omega_u$ the same difference is zero. Therefore,
\[
\int_M
\left(
\frac{\langle f(tu),tu\rangle_{h_\textbf{E}}}{t^q}
-
\frac{\langle f(su),su\rangle_{h_\textbf{E}}}{s^q}
\right)\,d\lambda_g
<0.
\]
Subtracting the two Nehari identities divided by $t^q$ and $s^q$, we get
\[
0\leq
\int_M
\Bigl(
a(t^p|\nabla^\textbf{E} u|_{g\otimes h_\textbf{E}}^p)t^{p-q}
-
a(s^p|\nabla^\textbf{E} u|_{g\otimes h_\textbf{E}}^p)s^{p-q}
\Bigr)
|\nabla^\textbf{E} u|_{g\otimes h_\textbf{E}}^p\,d\lambda_g\] \[=
\int_M
\left(
\frac{\langle f(tu),tu\rangle_{h_\textbf{E}}}{t^q}
-
\frac{\langle f(su),su\rangle_{h_\textbf{E}}}{s^q}
\right)\,d\lambda_g
<0,
\]
which is impossible. Therefore $t=s$. Hence the critical point $t_u$ is
unique.

It remains to determine the sign of $\gamma_u'$. Since $t_u$ is the only
critical point of $\gamma_u$, the continuity of $\gamma_u'$ implies that
$\gamma_u'$ has constant sign on each of the intervals $(0,t_u)$ and
$(t_u,+\infty)$. If $\gamma_u'<0$ on $(0,t_u)$, then $\gamma_u$ would be
strictly decreasing there. Since $\displaystyle\lim_{t\to0^+}\gamma_u(t)=I(0)=0$, this
would imply $\gamma_u(t)<0$ for every $t\in(0,t_u)$, contradicting the
positivity of $\gamma_u(t)$ for $t>0$ sufficiently small. Therefore
$\gamma_u'(t)>0$ for every $0<t<t_u$.

Similarly, if $\gamma_u'>0$ on $(t_u,+\infty)$, then $\gamma_u$ would be
strictly increasing there, contradicting
$\gamma_u(t)\to-\infty$ as $t\to+\infty$. Hence
$\gamma_u'(t)<0$ for every $t>t_u$. Thus $\gamma_u$ is strictly increasing
on $(0,t_u)$, strictly decreasing on $(t_u,+\infty)$, and $t_u$ is its
unique global maximum point.

Finally, let $u\in\mathcal N$. Then $t=1$ is a critical point of
$\gamma_u$, and by uniqueness it is the only critical point. From what we
have just proved, this point is the global maximum of $\gamma_u$. Since
$\gamma_u(t)>0$ for $t>0$ sufficiently small, we conclude that
$I(u)=\gamma_u(1)>0$. This completes the proof.
\end{proof}

In the following lemma, we will prove that any element of the Nehari manifold is bounded away from zero.
\begin{lemma}
\label{lem:nehari_away_from_zero}
There exists $C_0>0$ such that $\|u\|_X\geq C_0$ for every
$u\in\mathcal N$.
\end{lemma}

\begin{proof}
Let $u\in\mathcal N$. By the Nehari identity \eqref{nehari_identity}, we have

\begin{equation}\label{Imp}
\int_M
a(|\nabla^\textbf{E} u|_{g\otimes h_\textbf{E}}^{p})
|\nabla^\textbf{E} u|_{g\otimes h_\textbf{E}}^{p}\,d\lambda_g
=\int_M
\langle f(u),u\rangle_{h_\textbf{E}}\,d\lambda_g.
\end{equation}

Using the hypothesis \ref{Bi2a_1} on the left side of \eqref{Imp} and \eqref{estimative_bundle_f} of the Lemma~\ref{lem:fiberwise_growth_estimates} on the right side of \eqref{Imp} , we have
\begin{eqnarray}\label{Far_Zero1}
k_1\int_M|\nabla^\textbf{E} u|_{g\otimes h_\textbf{E}}^p\,d\lambda_g
+
k_2\int_M|\nabla^\textbf{E} u|_{g\otimes h_\textbf{E}}^q\,d\lambda_g \leq
\varepsilon\int_M|u|_{h_\textbf{E}}^p\,d\lambda_g
+C_\varepsilon\int_M|u|_{h_\textbf{E}}^r\,d\lambda_g.
\end{eqnarray}
By the Lemma \ref{Embedding}, there exist constant positive $C_1$ and $C_2$ such that
\begin{equation}\label{Imp1}
    \|u\|_{L^p(\textbf{E})}\leq
C_2\| u\|_{X} 
\end{equation}
and 
\begin{equation}\label{Imp2}
    \|u\|_{L^r(\textbf{E})} \leq
C_2\| u\|_{X} 
\end{equation}
Substituting \eqref{Imp1} and \eqref{Imp1} into \eqref{Imp}, we get

\begin{eqnarray}\label{Imp3}
k_2 \| u\|_{X}^q \leq
\varepsilon \|u\|_{L^p(\textbf{E})}^{p}
+C_\varepsilon \|u\|_{L^r(\textbf{E})}^{r} \leq \varepsilon C_2^p\| u\|_{X}^p
+C_\varepsilon C_1^r\| u\|_{X}^r
\end{eqnarray}
 \medskip\noindent \underline{Case $\| u\|_{X}>1$:} $2 \leq p \leq q$ and $q <r<q^*$
\begin{eqnarray*}
k_2 \| u\|_{X}^q  \leq \varepsilon C_2^p\| u\|_{X}^p
+C_\varepsilon C_1^r\| u\|_{X}^r \leq \max \{ \varepsilon C_2^p, C_\varepsilon C_1^r\}\| u\|_{X}^r 
\end{eqnarray*}
Consequently, we have
\begin{eqnarray}\label{Imp4}
0< \left(\frac{k_2}{ \max \{ \varepsilon C_2^p, C_\varepsilon C_1^r\}}\right)^{\frac{1}{r-q}} \leq\| u\|_{X}
\end{eqnarray}

\medskip
\noindent\underline{Case \(0<\|u\|_{X}<1\):}
Assume that \(2\le p\le q\) and \(q<r<q^*\). Since
\(\|u\|_{X}<1\), we have
\[
\|u\|_{X}^{q}\le \|u\|_{X}^{p}.
\]
Hence, from \eqref{Far_Zero1},
\[
k_{1}\|u\|_{X}^{p}
\le
k_{1}\|u\|_{X}^{p}
+
k_{2}\|u\|_{X}^{q}
\le
\varepsilon C_{2}^{p}\|u\|_{X}^{p}
+
C_{\varepsilon}C_{1}^{r}\|u\|_{X}^{r}.
\]
Choosing \(\varepsilon>0\) sufficiently small so that
\(k_{1}-\varepsilon C_{2}^{p}>0\), we obtain
\[
\bigl(k_{1}-\varepsilon C_{2}^{p}\bigr)\|u\|_{X}^{p}
\le
C_{\varepsilon}C_{1}^{r}\|u\|_{X}^{r},
\]
which implies
\begin{eqnarray}\label{Imp5}
0<   \left(
\frac{k_{1}-\varepsilon C_{2}^{p}}
{C_{\varepsilon}C_{1}^{r}}
\right)^{\frac1{r-p}} \leq \|u\|_{X} 
\end{eqnarray}
Combining \eqref{Imp4} and \eqref{Imp5} we get 
\begin{eqnarray}\label{Imp6}
0< C_0:= \min \left\{ \left(\frac{k_2}{ \max \{ \varepsilon C_2^p, C_\varepsilon C_1^r\}}\right)^{\frac{1}{r-q}} ,  \left(
\frac{k_{1}-\varepsilon C_{2}^{p}}
{C_{\varepsilon}C_{1}^{r}}
\right)^{\frac1{r-p}} \right\} \leq \|u\|_{X} 
\end{eqnarray}
which completes the proof.
\end{proof}

\begin{lemma}
\label{bounded_sequence}
If $(u_n)\subset\mathcal N$ is a minimizing sequence of $I$, then
$(u_n)$ is bounded in $X$.
\end{lemma}

\begin{proof}
Let \((u_n)\subset\mathcal N\) be a minimizing sequence for \(I\). Since
\(I(u_n)\to c\) and \(u_n\in\mathcal N\), we have
\[
I(u_n)=c+o_n(1)
\quad\text{and}\quad
\langle I'(u_n),u_n\rangle=0.
\]
Therefore,
\begin{eqnarray}
c+o_n(1)
&=&
I(u_n)-\frac1\theta\langle I'(u_n),u_n\rangle \nonumber\\
&=&
\frac1p
\int_M
A(|\nabla^{\mathbf E}u_n|_{g\otimes h_{\mathbf E}}^p)\,d\lambda_g
-
\frac1\theta
\int_M
a(|\nabla^{\mathbf E}u_n|_{g\otimes h_{\mathbf E}}^p)
|\nabla^{\mathbf E}u_n|_{g\otimes h_{\mathbf E}}^p\,d\lambda_g \nonumber\\
&&+
\int_M
\left(
\frac1\theta\langle f_x(u_n),u_n\rangle_{h_{\mathbf E}}
-\mathcal F_x(u_n)
\right)\,d\lambda_g.
\end{eqnarray}
By \eqref{eq:A_lower_from_monotonicity} and assumptions
\ref{Bi2a_1} and \ref{Bi2f_3},
\[
\frac1pA(t)-\frac1\theta a(t)t
\ge
\left(\frac1q-\frac1\theta\right)a(t)t,
\]
and
\[
\frac1\theta
\langle f_x(\xi),\xi\rangle
-\mathcal F_x(\xi)
\ge0.
\]
Hence,
\[
c+o_n(1)
\ge
\left(\frac1q-\frac1\theta\right)
\int_M
a(|\nabla^{\mathbf E}u_n|^p)
|\nabla^{\mathbf E}u_n|^p\,d\lambda_g.
\]
Using assumption \ref{Bi2a_1}, we obtain
\[
c+o_n(1)
\geq
\left(\frac1q-\frac1\theta\right)
\left(
k_1
\|\nabla^{\mathbf E}u_n\|_{L^p}^p
+
k_2
\|\nabla^{\mathbf E}u_n\|_{L^q}^q
\right).
\]
If \(p<q\), then \(k_2>0\), and hence
\[
c+o_n(1)
\geq
k_2\left(\frac1q-\frac1\theta\right)\|u_n\|_X^q.
\]
If \(p=q\), then \(k_1>0\) and
\[
c+o_n(1)
\geq
k_1\left(\frac1q-\frac1\theta\right)\|u_n\|_X^q.
\]
Thus, in both cases, \((u_n)\) is bounded in \(X\).
\end{proof}

\begin{proposition}
\label{prop:szulkin_weth_reduction}
Let \(S:=\{u\in X\mid \|u\|_X=1\}\). Then \(X\) and \(I\) satisfy
assumptions \((A_1)\)--\((A_3)\) of
Szulkin and Weth~\cite[Section~3.1]{SzulkinWethNehari}. Hence
\(\widehat m(u):=t_u u\) defines a continuous map
\(\widehat m:X\setminus\{0\}\to\mathcal N\), whose restriction
\(m:=\widehat m|_S:S\to\mathcal N\) is a homeomorphism with inverse
\(m^{-1}(u)=u/\|u\|_X\). Moreover,
\(\Psi:=I\circ m\in C^1(S,\mathbb R)\),
\(w\) is a critical point of \(\Psi\) if and only if
\(m(w)\) is a nontrivial critical point of \(I\), and
\[
\inf_{w\in S}\Psi=\inf_{u\in\mathcal N}I(u).
\]
\end{proposition}

\begin{proof}
It is well known that $X=W^{1,q}_0(\textbf{E})$, endowed with the norm
$\|\cdot\|_X$, is a uniformly convex real Banach space. Moreover,
Proposition~\ref{prop:energy_functional_C1} gives
$I\in C^1(X,\mathbb R)$, and $I(0)=0$. For $(A_1)$, take the normalization function
$\varphi(t):=t^{q-1}$. Then
\[
\psi(u)
:=
\int_0^{\|u\|_X}\varphi(t)\,dt
=
\frac1q\|u\|_X^q
\]
belongs to $C^1(X,\mathbb R)$, and its derivative $J:=\psi'$ is given by
\[
J(u)[v]
=
\int_M
|\nabla^\textbf{E} u|_{g\otimes h_\textbf{E}}^{q-2}
\langle\nabla^\textbf{E} u,\nabla^\textbf{E} v\rangle_{g\otimes h_\textbf{E}}
\,d\lambda_g.
\]
For every $u,v\in X$, the Cauchy--Schwarz inequality in each fiber of
$T^*M\otimes\textbf{E}$ first gives
\[
\left|
\left\langle\nabla^\textbf{E}u,\nabla^\textbf{E}v\right\rangle_{g\otimes h_\textbf{E}}
\right|
\leq
|\nabla^\textbf{E}u|_{g\otimes h_\textbf{E}}
|\nabla^\textbf{E}v|_{g\otimes h_\textbf{E}}.
\]
Therefore, we get
\[
|J(u)[v]|
\leq
\int_M
|\nabla^\textbf{E}u|_{g\otimes h_\textbf{E}}^{q-1}
|\nabla^\textbf{E}v|_{g\otimes h_\textbf{E}}\,d\lambda_g.
\]
Applying Hölder's inequality with conjugate exponents
$\frac{q}{q-1}$ and $q$, we obtain
\[
|J(u)[v]|\leq
\left(
\int_M
\bigl(|\nabla^\textbf{E}u|_{g\otimes h_\textbf{E}}^{q-1}\bigr)^{\frac{q}{q-1}}
\,d\lambda_g
\right)^{\frac{q-1}{q}}
\left(
\int_M
|\nabla^\textbf{E}v|_{g\otimes h_\textbf{E}}^q\,d\lambda_g
\right)^{\frac1q}=
\|u\|_X^{q-1}\|v\|_X.\]
By the definition of the dual norm,
\[
\|J(u)\|_{X'}
=
\sup_{\|v\|_X\leq1}|J(u)[v]|,\]
and hence
\[
\|J(u)\|_{X'}\leq\|u\|_X^{q-1}.
\]
If $u\neq0$, set $v_u:=u/\|u\|_X$. Then $\|v_u\|_X=1$, so $v_u$ is
admissible in the supremum defining $\|J(u)\|_{X^*}$, and
\[
\|J(u)\|_{X'}\geq
|J(u)[v_u]|=
\frac1{\|u\|_X}
\int_M|\nabla^\textbf{E}u|_{g\otimes h_\textbf{E}}^q\,d\lambda_g
=
\|u\|_X^{q-1}.
\]
Thus $\|J(u)\|_{X^*}=\|u\|_X^{q-1}$ for every $u\neq0$. In particular,
the preceding upper estimate shows directly that $J$ is bounded on bounded
subsets of $X$. Moreover, if $w\in S$, then
\[
J(w)[w]
=
\int_M|\nabla^\textbf{E}w|_{g\otimes h_\textbf{E}}^q\,d\lambda_g
=
\|w\|_X^q
=
1.
\]
Thus $(A_1)$ holds.

Assumption $(A_2)$ is precisely Lemma~\ref{Nehari_projection}. To verify
$(A_3)$, let $w\in S$. Since $t_w w\in\mathcal N$, Lemma
\ref{lem:nehari_away_from_zero} gives
\[
t_w=\|t_w w\|_X\geq C_0.
\]
It remains to prove the upper bound on compact subsets of $S$. Let
$K\subset S$ be compact. The assertion is immediate if $K=\varnothing$, so assume that $K$ is nonempty. By the upper estimate in \ref{Bi2a_1}, Hölder's inequality,
the finiteness of $\lambda_g(M)$, and the identity
$\|w\|_X=1$, there exists a constant $C>0$,
independent of $w\in S$ and $t>0$, such that
\[
\frac1p
\int_M
A(t^p|\nabla^\textbf{E}w|^p)\,d\lambda_g
\le
C(t^p+t^q).
\]
By Lemma~\ref{lem:lower_bound_F_bundle},
\[
I(tw)
\leq
C(t^p+t^q)
-Dt^\theta\|w\|_{L^\theta(\textbf{E})}^\theta
+D'\lambda_g(M).
\]
The map $w\mapsto\|w\|_{L^\theta(\textbf{E})}^\theta$ is continuous and positive
on $K$, so
\[
\mu_K:=\min_{w\in K}\|w\|_{L^\theta(\textbf{E})}^\theta>0.
\]
Since $\theta>q\geq p$, it follows that $I(tw)\to-\infty$ as
$t\to+\infty$, uniformly with respect to $w\in K$. Hence there
exists $T_K>0$ such that $I(tw)<0$ whenever $w\in K$ and $t\geq T_K$.
Since Lemma~\ref{Nehari_projection} gives $I(t_w w)>0$, necessarily
$t_w<T_K$ for every $w\in K$. Thus $(A_3)$ holds.

We may now apply the abstract results of Szulkin and Weth. By
\cite[Proposition~8(a)]{SzulkinWethNehari}, the map
$\widehat m:X\setminus\{0\}\to\mathcal N$ is continuous.  Their
Proposition~8(b) then identifies its restriction $m=\widehat m|_S$ as a
homeomorphism. In the present notation, the inverse can also be seen
directly from the uniqueness in Lemma~\ref{Nehari_projection}. Indeed, if
$u\in\mathcal N$ and $w:=u/\|u\|_X\in S$, then
$u=\|u\|_Xw$ is the unique point of the ray $\{tw\mid t>0\}$ which belongs to
$\mathcal N$. Hence $t_w=\|u\|_X$ and therefore $m(w)=u$. Conversely, for
$w\in S$,
\[
\frac{m(w)}{\|m(w)\|_X}
=
\frac{t_w w}{t_w\|w\|_X}
=
w.
\]
Thus $m^{-1}(u)=u/\|u\|_X$. Since the normalization map
$u\mapsto u/\|u\|_X$ is continuous on $X\setminus\{0\}$, this also gives
directly the continuity of the inverse.

Finally, if $\widehat\Psi:=I\circ\widehat m$, then
\cite[Proposition~9]{SzulkinWethNehari} gives
$\widehat\Psi\in C^1(X\setminus\{0\},\mathbb R)$ and
\[
\langle\widehat\Psi'(u),v\rangle
=
\frac{\|\widehat m(u)\|_X}{\|u\|_X}
\langle I'(\widehat m(u)),v\rangle,
\qquad u\in X\setminus\{0\},\quad v\in X.
\]
Restricting to $S$ and using
\cite[Corollary~10(a),(c)]{SzulkinWethNehari}, we obtain
$\Psi=I\circ m\in C^1(S,\mathbb R)$, the stated correspondence between
critical points of $\Psi$ and nontrivial critical points of $I$, and the
identity
\[
\inf_{w\in S}\Psi=\inf_{u\in\mathcal N}I(u).
\]
\end{proof}

In the following lemma, we will prove the existence of a minimizer on the Nehari manifold.
\begin{lemma}
\label{lem:nehari_minimizer}
There exists \(u_0\in\mathcal N\) such that $\displaystyle I(u_0)=\min_{u\in\mathcal N}I(u)=c$.
\end{lemma}

\begin{proof}
Let \((u_n)\subset\mathcal N\) be a minimizing sequence for \(I\), so
that \(I(u_n)\to c\). By Lemma~\ref{lem:nonlinear_convergence}, after
passing to a subsequence there exists \(u\in X\) such that
\(u_n\rightharpoonup u\) weakly in \(X\),
\(u_n\to u\) strongly in \(L^p(\textbf{E})\cap L^r(\textbf{E})\), and
\eqref{eq:nonlinear_convergence_f_scaled}--\eqref{eq:nonlinear_convergence_F_scaled}
hold for every fixed \(t\geq0\).

We first prove that \(u\neq0\). Suppose, by contradiction, that \(u=0\). 
Taking \(t=1\) in \eqref{eq:nonlinear_convergence_f_scaled}, we obtain
\[
\int_M
\langle f_x(u_n),u_n\rangle_{h_{\mathbf E}}\,d\lambda_g
\longrightarrow0.
\]
Using the Nehari identity \eqref{nehari_identity} and assumption
\ref{Bi2a_1},
\[k_1\int_M
|\nabla^\textbf{E}u_n|_{g\otimes h_\textbf{E}}^p\,d\lambda_g
+
k_2\int_M
|\nabla^\textbf{E}u_n|_{g\otimes h_\textbf{E}}^q\,d\lambda_g\leq
\int_M
\langle f_x(u_n),u_n\rangle_{h_\textbf{E}}\,d\lambda_g
\longrightarrow0.
\]

\medskip
\noindent\underline{If \(p<q\)}, then \(k_2>0\), and therefore \(\|u_n\|_X\to0\).

\medskip
\noindent\underline{If \(p=q\)}, then \(k_1>0\), and the same conclusion follows from the
first integral. In either case this contradicts
Lemma~\ref{lem:nehari_away_from_zero}. Hence \(u\neq0\).

\medskip

By Lemma~\ref{Nehari_projection}, there exists a unique \(t_u>0\) such
that \(t_u u\in\mathcal N\). Since \(u_n\in\mathcal N\), the same lemma
shows that \(t=1\) is the global maximum of \(t\mapsto I(tu_n)\) on
\((0,+\infty)\). Consequently, $I(t_u u_n)\leq I(u_n)$ for every $n$.

Recall that
\[
J_1(v):=
\frac1p
\int_M
A(|\nabla^\textbf{E}v|_{g\otimes h_\textbf{E}}^p)\,d\lambda_g.
\]
By \ref{Bi2a_3}, the function \(G(t):=A(t^p)\) is convex on
\([0,+\infty)\). Since \(a\geq0\), both \(A\) and \(G\) are
nondecreasing. Hence the map
\[
\eta\longmapsto
\frac1pA(|\eta|_{g\otimes h_\textbf{E}}^p)
\]
is convex on each fiber of \(T^*M\otimes\textbf{E}\), and therefore
\(J_1\) is convex on \(X\). By
Proposition~\ref{prop:energy_functional_C1}, \(J_1\) is continuous and
thus weakly lower semicontinuous. Since
\(t_u u_n\rightharpoonup t_u u\) weakly in \(X\),
\[
J_1(t_u u)
\leq
\liminf_{n\to\infty}J_1(t_u u_n).
\]
On the other hand, taking \(t=t_u\) in
\eqref{eq:nonlinear_convergence_F_scaled} gives
\[
\int_M\mathcal F_x(t_u u_n)\,d\lambda_g
\longrightarrow
\int_M\mathcal F_x(t_u u)\,d\lambda_g.
\]
Combining these two facts, we obtain
\[
I(t_u u)
\leq
\liminf_{n\to\infty}I(t_u u_n).
\]
Since \(t_u u\in\mathcal N\), the definition of \(c\) gives
\(c\leq I(t_u u)\). Moreover, \(I(t_u u_n)\leq I(u_n)\) for every
\(n\). Therefore,
\[
c
\leq
I(t_u u)
\leq 
\liminf_{n\to\infty}I(t_u u_n)
\leq
\lim_{n\to\infty}I(u_n)
=
c.
\]
Thus \(I(t_u u)=c\). Setting \(u_0:=t_u u\), we conclude that
\(u_0\in\mathcal N\) and
\[
I(u_0)=c=\min_{v\in\mathcal N}I(v).
\]
\end{proof}

Finally, we can establish the demonstration of our main variational result, Theorem \ref{teo_A}.

\begin{proof}[Proof of Theorem~\ref{teo_A}]
By Lemma~\ref{lem:nehari_minimizer}, there exists $u_0\in\mathcal N$ such
that $I(u_0)=c$. Let $w_0:=m^{-1}(u_0)\in S$, where $m$ is the
homeomorphism from Proposition~\ref{prop:szulkin_weth_reduction}. Then
\[
\Psi(w_0)
=
I(u_0)
=
c
=
\inf_{w\in S}\Psi(w).
\]
Thus $w_0$ is a critical point of $\Psi$. By the critical-point
correspondence in Proposition~\ref{prop:szulkin_weth_reduction},
$u_0=m(w_0)$ is a nontrivial critical point of $I$. Proposition
\ref{prop:energy_functional_C1} therefore shows that $u_0$ is a
nontrivial weak solution of \eqref{ProblemP}.

If $v\in X\setminus\{0\}$ is any other critical point of $I$, then
$\langle I'(v),v\rangle=0$, and hence $v\in\mathcal N$. Therefore
\[
I(u_0)=c\leq I(v).
\]
Consequently, $u_0$ is a ground state solution of \eqref{ProblemP}.
\end{proof}

\appendix

\section{Appendix}\label{sec:appendix}

The next lemma provides the continuity property for fiberwise superposition operators that will be used in the proof of the differentiability of the energy functional.
\begin{lemma}\label{lem:fiberwise_superposition_continuity}
Let $F\to M$ and $G\to M$ be finite-rank metric vector bundles, and let
$\mathscr G:F\to G$ be a continuous fiber-preserving map. Suppose that
$m>1$ and that there exists $C>0$ such that
$\displaystyle |\mathscr G_x(\xi)|_G\leq C(1+|\xi|_F^{m-1})$
for every $x\in M$ and every $\xi\in F_x$. Then the fiberwise
superposition operator
$\displaystyle \mathcal N_{\mathscr G}:L^m(F)\to L^{m'}(G)$,
defined by
$\displaystyle \mathcal N_{\mathscr G}(w)(x):=\mathscr G_x(w(x))$,
where $ m':=\frac{m}{m-1}$, is well defined and continuous.
\end{lemma}

\begin{proof}
Since $(m-1)m'=m$, the growth assumption gives
$\displaystyle |\mathscr G_x(w(x))|_G^{m'}
\leq C(1+|w(x)|_F^m)$ almost everywhere. Since $M$ has finite measure,
$\mathscr G(w)\in L^{m'}(G)$ whenever $w\in L^m(F)$.

Let $w_k\to w$ in $L^m(F)$ and let $(w_{k_\ell})$ be an arbitrary
subsequence. Passing to a further subsequence, we may assume that
$\displaystyle \|w_{k_\ell}-w\|_{L^m(F)}\leq2^{-\ell}$. The function
$ H:=\sum_{\ell=1}^{\infty}|w_{k_\ell}-w|_F$ belongs to
$L^m(M)$, because
$ \|H\|_{L^m}\leq\sum_{\ell=1}^{\infty}2^{-\ell}$.     
Consequently, $w_{k_\ell}(x)\to w(x)$ and
$\displaystyle |w_{k_\ell}(x)|_F\leq |w(x)|_F+H(x)$ almost everywhere.

By continuity of $\mathscr G$, one has
$\mathscr G_x(w_{k_\ell}(x))\to\mathscr G_x(w(x))$ almost everywhere.
Moreover,
\[
|\mathscr G_x(w_{k_\ell}(x))-\mathscr G_x(w(x))|_G^{m'}
\leq
C\bigl(1+H(x)^m+|w(x)|_F^m\bigr).
\]
The right-hand side belongs to $L^1(M)$, so the dominated convergence
theorem yields
$\displaystyle \mathscr G(w_{k_\ell})\to\mathscr G(w)$ in $L^{m'}(G)$.
Since every subsequence has a further subsequence with this property,
the whole sequence converges. Thus $\mathcal N_{\mathscr G}$ is
continuous.
\end{proof}

\begin{lemma}\label{lem:nonlinear_convergence}
Let \((u_n)\subset\mathcal N\) be a minimizing sequence for \(c\). Then,
after passing to a subsequence, there exists \(u\in X\) such that
\(u_n\rightharpoonup u\) weakly in \(X\) and
\(u_n\to u\) strongly in
\(L^p(\textbf{E})\cap L^r(\textbf{E})\). Moreover, for every fixed
\(t\geq0\),
\begin{equation}\label{eq:nonlinear_convergence_f_scaled}
\lim_{n\to\infty}
\int_M
\langle f_x(tu_n),tu_n\rangle_{h_{\mathbf E}}\,d\lambda_g
=
\int_M
\langle f_x(tu),tu\rangle_{h_{\mathbf E}}\,d\lambda_g,
\end{equation}
and
\begin{equation}\label{eq:nonlinear_convergence_F_scaled}
\lim_{n\to\infty}
\int_M
\mathcal F_x(tu_n)\,d\lambda_g
=
\int_M
\mathcal F_x(tu)\,d\lambda_g.
\end{equation}
\end{lemma}

\begin{proof}
By Lemma~\ref{bounded_sequence}, \((u_n)\) is bounded in \(X\). Since
\(X=W^{1,q}_0(\textbf{E})\) is reflexive, after passing to a subsequence
there exists \(u\in X\) such that \(u_n\rightharpoonup u\) weakly in
\(X\). By the compact Sobolev embedding in Lemma~\ref{Embedding}, and
since \(p\leq q<r<q^*\), after passing to a further subsequence if
necessary,
\[
u_n\longrightarrow u
\qquad\text{strongly in }L^p(\textbf{E})\cap L^r(\textbf{E}).
\]
Fix \(t\geq0\). If \(t=0\), both conclusions are immediate. Suppose
therefore that \(t>0\).

We first prove \eqref{eq:nonlinear_convergence_f_scaled}. Since
\(tu_n\to tu\) strongly in \(L^r(\textbf{E})\), assumption~\ref{Bi2f_2}
and Lemma~\ref{lem:fiberwise_superposition_continuity}, applied with
\(m=r\) and \(\mathscr G=f\), yield $f(tu_n)\longrightarrow f(tu)$ strongly in $L^{r'}(\textbf{E})$, where $r'=\frac{r}{r-1}$. Moreover, \((tu_n)\) is bounded in \(L^r(\textbf{E})\). Hence, by
H\"older's inequality,
\[
\begin{aligned}
&\left|
\int_M\langle f_x(tu_n),tu_n\rangle_{h_{\mathbf E}}\,d\lambda_g
-
\int_M\langle f_x(tu),tu\rangle_{h_{\mathbf E}}\,d\lambda_g
\right|\\
&\qquad\leq
\|f(tu_n)-f(tu)\|_{L^{r'}(\textbf{E})}
\|tu_n\|_{L^r(\textbf{E})}
+
\|f(tu)\|_{L^{r'}(\textbf{E})}
\|t(u_n-u)\|_{L^r(\textbf{E})},
\end{aligned}
\]
and the right-hand side converges to zero. This proves
\eqref{eq:nonlinear_convergence_f_scaled}.

We now prove \eqref{eq:nonlinear_convergence_F_scaled}. For almost every
\(x\in M\), the points \(u_n(x)\) and \(u(x)\) belong to the same fiber
\(\textbf{E}_x\). Since \(\mathcal F_x\in C^1(\textbf{E}_x)\) and
\(f_x\) is its fiber gradient, the fundamental theorem of calculus on
the segment joining \(tu(x)\) and \(tu_n(x)\) gives
\begin{eqnarray*}
    \mathcal F_x(tu_n(x))-\mathcal F_x(tu(x))=
t\int_0^1
\left\langle
f_x\!\left(t\bigl((1-s)u(x)+su_n(x)\bigr)\right),
u_n(x)-u(x)
\right\rangle_{h_{\mathbf E}}\,ds.
\end{eqnarray*}
For \(s\in[0,1]\), set
\(w_{n,s}:=t\bigl((1-s)u+su_n\bigr)\). Since \((u_n)\) is bounded in
\(L^r(\textbf{E})\), there exists \(R_t>0\), independent of \(n\) and
\(s\), such that
\[
\|w_{n,s}\|_{L^r(\textbf{E})}
\leq
t\bigl((1-s)\|u\|_{L^r(\textbf{E})}
+s\|u_n\|_{L^r(\textbf{E})}\bigr)
\leq R_t.
\]
By assumption~\ref{Bi2f_2}, the identity \((r-1)r'=r\), and the
finiteness of \(\lambda_g(M)\), there is a constant \(C_t>0\),
independent of \(n\) and \(s\), such that
\begin{eqnarray*}
\|f(w_{n,s})\|_{L^{r'}(\textbf{E})}^{r'} \leq  C\int_M
\bigl(1+|w_{n,s}|_{h_{\mathbf E}}^{r-1}\bigr)^{r'}\,d\lambda_g   \leq
C\left(\lambda_g(M)+\|w_{n,s}\|_{L^r(\textbf{E})}^r\right)
\leq
C\left(\lambda_g(M)+R_t^r\right),
\end{eqnarray*}
and consequently \(\|f(w_{n,s})\|_{L^{r'}(\textbf{E})}\leq C_t\).
In particular,
\[
\int_0^1\int_M
|f(w_{n,s})|_{h_{\mathbf E}}|u_n-u|_{h_{\mathbf E}}
\,d\lambda_g\,ds
\leq
C_t\|u_n-u\|_{L^r(\textbf{E})}<\infty,
\]
so Fubini's theorem is applicable. Using the preceding identity,
Fubini's theorem, and H\"older's inequality, we obtain
\begin{eqnarray*}
  \left|
\int_M\mathcal F_x(tu_n)\,d\lambda_g
-
\int_M\mathcal F_x(tu)\,d\lambda_g
\right|  \leq
t\|u_n-u\|_{L^r(\textbf{E})}
\int_0^1\|f(w_{n,s})\|_{L^{r'}(\textbf{E})}\,ds \leq
tC_t\|u_n-u\|_{L^r(\textbf{E})}
\to 0.
\end{eqnarray*}
This proves \eqref{eq:nonlinear_convergence_F_scaled}.
\end{proof}

\begin{proposition}\label{prop:energy_functional_C1}
The functional $I:X\to\mathbb R$ is well defined and belongs to
$C^1(X,\mathbb R)$. Moreover, for every $u,v\in X$,
\begin{equation}\label{energy_derivative}
\langle I'(u),v\rangle=
\int_M
a(|\nabla^\textbf{E} u|_{g\otimes h_\textbf{E}}^{p})
|\nabla^\textbf{E} u|_{g\otimes h_\textbf{E}}^{p-2}
\langle\nabla^\textbf{E} u,\nabla^\textbf{E} v\rangle_{g\otimes h_\textbf{E}}
\,d\lambda_g-
\int_M
\langle f_x(u(x)),v(x)\rangle_{h_\textbf{E}}
\,d\lambda_g.
\end{equation}
Consequently, the critical points of $I$ are precisely the weak
solutions of \eqref{ProblemP}.
\end{proposition}

\begin{proof}
Write $I=J_1-J_2$, where

\begin{equation*}
\displaystyle J_1(u):=\frac1p\int_M A(|\nabla^\textbf{E} u|_{g\otimes h_\textbf{E}}^p)\,d\lambda_g    
\end{equation*}
and
\begin{equation*}
\displaystyle J_2(u):=\int_M\mathcal F_x(u(x))\,d\lambda_g
\end{equation*}

We first show that $I$ is well defined. By hypothesis \ref{Bi2a_1}, one has
$\displaystyle a(\rho^p)\leq k_3+k_4\rho^{q-p}$ for every $\rho>0$.
Therefore
\begin{equation*}
    \displaystyle A(\rho^p)=p\int_0^\rho a(\tau^p)\tau^{p-1}\,d\tau
\leq k_3\rho^p+\frac{pk_4}{q}\rho^q.
\end{equation*}
Since $\nabla^\textbf{E} u\in L^q(T^*M\otimes \textbf{E})$ and $p\leq q$, it follows that
$\displaystyle A(|\nabla^\textbf{E} u|_{g\otimes h_\textbf{E}}^p)\in L^1(M)$.
Moreover, Lemma~\ref{lem:fiberwise_growth_estimates} gives
$\displaystyle |\mathcal F_x(\xi)|\leq C(|\xi|_{h_\textbf{E}}^p+|\xi|_{h_\textbf{E}}^r)$.
Since $X\hookrightarrow L^p(\textbf{E})\cap L^r(\textbf{E})$ continuously, we also have
$\displaystyle \mathcal F_x(u(x))\in L^1(M)$. Hence $I$ is well defined.

We now prove that $J_1\in C^1(X,\mathbb R)$. Let $V:=T^*M\otimes \textbf{E}$ and
define $\Psi:V\to V$ fiberwise by
\[
\Psi_x(\eta):=
a(|\eta|_{g\otimes h_\textbf{E}}^p)
|\eta|_{g\otimes h_\textbf{E}}^{p-2}\eta,
\qquad
\Psi_x(0_x):=0_x.
\]
Also set
$\displaystyle \Phi_x(\eta):=\frac1p A(|\eta|_{g\otimes h_\textbf{E}}^p)$.
Then
$\displaystyle D\Phi_x(\eta)[\zeta]
=\langle\Psi_x(\eta),\zeta\rangle_{g\otimes h_\textbf{E}}$.
Furthermore, by \ref{Bi2a_1},
\[
|\Psi_x(\eta)|_{g\otimes h_\textbf{E}}
\leq
k_3|\eta|_{g\otimes h_\textbf{E}}^{p-1}
+k_4|\eta|_{g\otimes h_\textbf{E}}^{q-1}
\leq
C\bigl(1+|\eta|_{g\otimes h_\textbf{E}}^{q-1}\bigr),
\]
where in the last inequality we used $p\leq q$. Thus, by
Lemma~\ref{lem:fiberwise_superposition_continuity}, applied with $m=q$,
the superposition operator induced by $\Psi$ is continuous from $L^q(V)$
into $L^{q'}(V)$, where $\displaystyle q':=\frac{q}{q-1}$.

For $u\in X$, define $L_u\in X'$ by
$\displaystyle \langle L_u,v\rangle:=
\int_M\langle\Psi_x(\nabla^\textbf{E} u),\nabla^\textbf{E} v\rangle_{g\otimes h_\textbf{E}}
\,d\lambda_g$.
This is well defined and continuous, because
\[
|\langle L_u,v\rangle|
\leq
\|\Psi(\nabla^\textbf{E} u)\|_{L^{q'}(V)}
\|\nabla^\textbf{E} v\|_{L^q(V)}
\leq
C\|\Psi(\nabla^\textbf{E} u)\|_{L^{q'}(V)}\|v\|_X.
\]

Let $h\in X$. For almost every $x\in M$, applying the fundamental theorem
of calculus to
$s\mapsto \Phi_x(\nabla^\textbf{E} u(x)+s\nabla^\textbf{E} h(x))$ gives
\[
\Phi_x(\nabla^\textbf{E} u+\nabla^\textbf{E} h)-\Phi_x(\nabla^\textbf{E} u)
=
\int_0^1
\left\langle
\Psi_x(\nabla^\textbf{E} u+s\nabla^\textbf{E} h),
\nabla^\textbf{E} h
\right\rangle_{g\otimes h_\textbf{E}}
\,ds.
\]
Integrating this identity over $M$ and subtracting the identity
$\langle L_u,h\rangle=\displaystyle\int_M\int_0^1
\langle\Psi_x(\nabla^\textbf{E} u),\nabla^\textbf{E} h\rangle_{g\otimes h_\textbf{E}}
\,ds\,d\lambda_g$, we obtain
\[
\begin{aligned}
J_1(u+h)-J_1(u)-\langle L_u,h\rangle=
\int_M\int_0^1
\left\langle
\Psi_x(\nabla^\textbf{E} u+s\nabla^\textbf{E} h)-\Psi_x(\nabla^\textbf{E} u),
\nabla^\textbf{E} h
\right\rangle_{g\otimes h_\textbf{E}}
\,ds\,d\lambda_g.
\end{aligned}
\]
Applying Fubini's theorem and then Hölder's inequality, we obtain
\[
\begin{aligned}
|J_1(u+h)-J_1(u)-\langle L_u,h\rangle| &\leq
\|\nabla^\textbf{E} h\|_{L^q(V)}
\int_0^1
\|\Psi(\nabla^\textbf{E} u+s\nabla^\textbf{E} h)-\Psi(\nabla^\textbf{E} u)\|_{L^{q'}(V)}
\,ds\\
&\quad\leq
C\|h\|_X
\int_0^1
\|\Psi(\nabla^\textbf{E} u+s\nabla^\textbf{E} h)-\Psi(\nabla^\textbf{E} u)\|_{L^{q'}(V)}
\,ds.
\end{aligned}
\]
If $h_j\to0$ in $X$, then
$\displaystyle \nabla^\textbf{E} u+s\nabla^\textbf{E} h_j\to\nabla^\textbf{E} u$ in $L^q(V)$ for
every $s\in[0,1]$. By the continuity of the superposition operator
induced by $\Psi$,
\[
\Psi(\nabla^\textbf{E} u+s\nabla^\textbf{E} h_j)
\longrightarrow
\Psi(\nabla^\textbf{E} u)
\quad\text{in }L^{q'}(V)
\]
for every $s\in[0,1]$. Moreover,
\[
\|\Psi(\nabla^\textbf{E} u+s\nabla^\textbf{E} h_j)-\Psi(\nabla^\textbf{E} u)\|_{L^{q'}(V)}\leq
\|\Psi(\nabla^\textbf{E} u+s\nabla^\textbf{E} h_j)\|_{L^{q'}(V)}
+
\|\Psi(\nabla^\textbf{E} u)\|_{L^{q'}(V)}.
\]
Since $h_j\to0$ in $X$, the sequence
$\displaystyle \nabla^\textbf{E} u+s\nabla^\textbf{E} h_j$ is bounded in $L^q(V)$ uniformly
with respect to $s\in[0,1]$ and $j$. Using the growth estimate for
$\Psi$, the preceding right-hand side is bounded by a constant independent
of $s$ and $j$. Hence, by the dominated convergence theorem in the
variable $s$,
\[
\int_0^1
\|\Psi(\nabla^\textbf{E} u+s\nabla^\textbf{E} h_j)-\Psi(\nabla^\textbf{E} u)\|_{L^{q'}(V)}
\,ds
\longrightarrow0.
\]
Therefore
\[
\lim_{\|h\|_X\to0}
\frac{|J_1(u+h)-J_1(u)-\langle L_u,h\rangle|}{\|h\|_X}
=0.
\]
Hence $J_1$ is Fréchet differentiable at $u$ and $J_1'(u)=L_u$, that is,
\[
\langle J_1'(u),v\rangle
=
\int_M
a(|\nabla^\textbf{E} u|_{g\otimes h_\textbf{E}}^p)
|\nabla^\textbf{E} u|_{g\otimes h_\textbf{E}}^{p-2}
\langle\nabla^\textbf{E} u,\nabla^\textbf{E} v\rangle_{g\otimes h_\textbf{E}}
\,d\lambda_g.
\]

If $u_k\to u$ in $X$, then
$\nabla^\textbf{E} u_k\to\nabla^\textbf{E} u$ in $L^q(V)$. By the continuity of the
superposition operator induced by $\Psi$,
$\Psi(\nabla^\textbf{E} u_k)\to\Psi(\nabla^\textbf{E} u)$ in $L^{q'}(V)$. Therefore
\[
\|J_1'(u_k)-J_1'(u)\|_{X'}
\leq
C\|\Psi(\nabla^\textbf{E} u_k)-\Psi(\nabla^\textbf{E} u)\|_{L^{q'}(V)}
\longrightarrow0.
\]
Thus $J_1\in C^1(X,\mathbb R)$.

We now prove that $J_2\in C^1(X,\mathbb R)$. By
Lemma~\ref{lem:fiberwise_growth_estimates},
\[
|f_x(\xi)|_{h_\textbf{E}}
\leq
C\bigl(|\xi|_{h_\textbf{E}}^{p-1}+|\xi|_{h_\textbf{E}}^{r-1}\bigr)
\leq
C\bigl(1+|\xi|_{h_\textbf{E}}^{r-1}\bigr),
\]
where we used $p<r$. Hence, by
Lemma~\ref{lem:fiberwise_superposition_continuity}, applied with $m=r$,
the superposition operator induced by $f$ is continuous from $L^r(\textbf{E})$
into $L^{r'}(\textbf{E})$, where $\displaystyle r':=\frac{r}{r-1}$.

For $u\in X$, define $K_u\in X'$ by
$\displaystyle \langle K_u,v\rangle:=
\int_M\langle f_x(u(x)),v(x)\rangle_{h_\textbf{E}}\,d\lambda_g$.
Since $X\hookrightarrow L^r(\textbf{E})$ continuously, this functional is well
defined and
\[
|\langle K_u,v\rangle|
\leq
\|f(u)\|_{L^{r'}(\textbf{E})}\|v\|_{L^r(\textbf{E})}
\leq
C\|f(u)\|_{L^{r'}(\textbf{E})}\|v\|_X.
\]

Let $h\in X$. For almost every $x\in M$, applying the fundamental theorem
of calculus to $s\mapsto\mathcal F_x(u(x)+sh(x))$ gives
\[
\mathcal F_x(u+h)-\mathcal F_x(u)
=
\int_0^1
\langle f_x(u+sh),h\rangle_{h_\textbf{E}}
\,ds.
\]
Integrating in $M$ and subtracting
$\displaystyle \langle K_u,h\rangle
=\int_M\int_0^1\langle f_x(u),h\rangle_{h_\textbf{E}}\,ds\,d\lambda_g$, we obtain
\[
\begin{aligned}
J_2(u+h)-J_2(u)-\langle K_u,h\rangle=
\int_M\int_0^1
\langle f_x(u(x)+sh(x))-f_x(u(x)),h(x)\rangle_{h_\textbf{E}}
\,ds\,d\lambda_g.
\end{aligned}
\]
Therefore, by Fubini's theorem and Hölder's inequality and the embedding $X\hookrightarrow L^r(\textbf{E})$,
\[
\begin{aligned}
|J_2(u+h)-J_2(u)-\langle K_u,h\rangle|
&\quad\leq
\|h\|_{L^r(\textbf{E})}
\int_0^1
\|f(u+sh)-f(u)\|_{L^{r'}(\textbf{E})}
\,ds\\
&\quad\leq
C\|h\|_X
\int_0^1
\|f(u+sh)-f(u)\|_{L^{r'}(\textbf{E})}
\,ds.
\end{aligned}
\]
If $h_j\to0$ in $X$, then $u+sh_j\to u$ in $L^r(\textbf{E})$ for every
$s\in[0,1]$. By the continuity of the superposition operator induced by
$f$,
$\displaystyle f(u+sh_j)\to f(u)$ in $L^{r'}(\textbf{E})$ for every $s\in[0,1]$.
Moreover, since $h_j\to0$ in $X$, the sections $u+sh_j$ are bounded in
$L^r(\textbf{E})$ uniformly with respect to $s\in[0,1]$ and $j$. Using the growth
estimate for $f$, the quantities
$\displaystyle \|f(u+sh_j)-f(u)\|_{L^{r'}(\textbf{E})}$ are bounded by a constant
independent of $s$ and $j$. Hence, by the dominated convergence theorem
in the variable $s$,
\[
\int_0^1
\|f(u+sh_j)-f(u)\|_{L^{r'}(\textbf{E})}
\,ds
\longrightarrow0.
\]
Therefore
\[
\lim_{\|h\|_X\to0}
\frac{|J_2(u+h)-J_2(u)-\langle K_u,h\rangle|}{\|h\|_X}
=0.
\]
Thus $J_2$ is Fréchet differentiable at $u$ and $J_2'(u)=K_u$, that is,
\[
\langle J_2'(u),v\rangle
=
\int_M
\langle f_x(u(x)),v(x)\rangle_{h_\textbf{E}}
\,d\lambda_g.
\]

If $u_k\to u$ in $X$, then $u_k\to u$ in $L^r(\textbf{E})$. By the continuity of
the superposition operator induced by $f$, $f(u_k)\to f(u)$ in
$L^{r'}(\textbf{E})$. Therefore
\[
\|J_2'(u_k)-J_2'(u)\|_{X'}
\leq
C\|f(u_k)-f(u)\|_{L^{r'}(\textbf{E})}
\longrightarrow0.
\]
Thus $J_2\in C^1(X,\mathbb R)$.

It follows that $I=J_1-J_2\in C^1(X,\mathbb R)$, and
\eqref{energy_derivative} follows from the expressions for $J_1'$ and
$J_2'$. Finally, $I'(u)=0$ if and only if
\[
\int_M
a(|\nabla^\textbf{E} u|_{g\otimes h_\textbf{E}}^p)
|\nabla^\textbf{E} u|_{g\otimes h_\textbf{E}}^{p-2}
\langle\nabla^\textbf{E} u,\nabla^\textbf{E} v\rangle_{g\otimes h_\textbf{E}}
\,d\lambda_g
=
\int_M
\langle f_x(u(x)),v(x)\rangle_{h_\textbf{E}}
\,d\lambda_g
\]
for every $v\in X$, which is precisely the weak formulation of
\eqref{ProblemP}.
\end{proof}

\section*{Declarations}

\begin{itemize}

\item \textbf{Funding.}
No funding was received for this work.

\item \textbf{Conflict of interest.}
The authors declare that they have no conflict of interest.

\item \textbf{Ethics approval.}
Not applicable.

%\item \textbf{Consent for publication.}
%Not applicable.

\item \textbf{Data availability.}
Not applicable.

%\item \textbf{Materials availability.}
%Not applicable.

%\item \textbf{Code availability.}
%Not applicable.

\item \textbf{Author contributions.}
All authors contributed to the conception, mathematical development, and writing of this work. All authors participated in the revision of the manuscript, read it, and approved the final version.

\end{itemize}


\begin{thebibliography}{99}
\bibitem{SzulkinWethNehari}
A.~Szulkin and T.~Weth,
\textit{The Method of Nehari Manifold},
in D.~Y.~Gao and D.~Motreanu (eds.),
\textit{Handbook of Nonconvex Analysis and Applications},
International Press, Boston, 2010, pp.~597--632.


\bibitem{BehzadanHolst}
A.~Behzadan and M.~Holst,
\textit{Sobolev--Slobodeckij spaces on compact manifolds, revisited},
Mathematics \textbf{10} (2022), no.~3, Article~522.

\bibitem{VelazquezSandoval}
C.~D.~Velázquez-Mendoza and M.~A.~Sandoval-Romero,
\textit{Geometric Integration by Parts and Sobolev Spaces on Vector Bundles:
A Unified Global Approach},
\texttt{arXiv.2602.01016} Accepted at Nonlinear Differential Equations and Applications NoDEA. 2026.


\bibitem{Aubin}
T.~Aubin,
\textit{Some Nonlinear Problems in Riemannian Geometry},
Springer-Verlag, Berlin Heidelberg, 1998.


\bibitem{Hebey1}
E.~Hebey,
\textit{Nonlinear Analysis on Manifolds: Sobolev Spaces and Inequalities},
Courant Institute of Mathematical Sciences, New York, 1999.

\bibitem{BoossWojciechowski}
B.~Booß-Bavnbek and K.~P.~Wojciechowski,
\textit{Elliptic Boundary Problems for Dirac Operators},
Mathematics: Theory \& Applications,
Birkhäuser, Boston, 1993.

\bibitem{Guneysu}
B.~Güneysu,
\textit{Covariant Schrödinger Semigroups on Riemannian Manifolds},
Operator Theory: Advances and Applications, vol.~264,
Birkhäuser, Cham, 2017.

%\bibitem{AlvesFigueiredo} C. O. Alves and G. M. Figueiredo, {\it Multiplicity and Concentration of Positive Solutions for a Class of Quasilinear Problems}, Adv. Nonlinear Stud. {11} (2011), 265-294.


\bibitem{AlvesFigueiredo}
C.~O.~Alves and G.~M.~Figueiredo,
\textit{Multiplicity and concentration of positive solutions for a class of quasilinear problems}, Adv. Nonlinear Stud. \textbf{11} (2011), no.~2, 265--294.



\bibitem{NolderRyan}
C.~A.~Nolder and J.~Ryan,
\textit{\(p\)-Dirac operators},
Adv. Appl. Clifford Algebr. \textbf{19} (2009), no.~2, 391--402.

\bibitem{KohrNistor}
M.~Kohr and V.~Nistor,
\textit{Sobolev spaces and \(\nabla\)-differential operators on manifolds I:
Basic properties and weighted spaces},
Ann. Global Anal. Geom. \textbf{61} (2022), no.~4, 721--758.


\bibitem{GrosseSchneider}
N.~Große and C.~Schneider,
\textit{Sobolev spaces on Riemannian manifolds with bounded geometry:
General coordinates and traces},
Math. Nachr. \textbf{286} (2013), no.~16, 1586--1613.

\bibitem{VanSchaftingenWinter}
J.~Van Schaftingen and L.~Winter,
\textit{Trace theory for gauge-covariant Sobolev spaces},
J. Math. Anal. Appl. \textbf{551} (2025), no.~2,
Article~129697.

\bibitem{Leoni}
G.~Leoni,
\textit{A First Course in Sobolev Spaces},
2nd ed., Graduate Studies in Mathematics, vol.~181,
American Mathematical Society, Providence, RI, 2017.

\bibitem{Lieberman1991}
G. M. Lieberman,
{\it The Natural Generalization of the Natural Conditions of Ladyzhenskaya and Ural'tseva for Elliptic Equations},
Comm. Partial Differential Equations {\bf 16} (1991), no. 2--3, 311--361.


\bibitem{Costa2023}
G.~S.~A. Costa,
\newblock Existence of solutions for a class of quasilinear equations with vanishing potentials,
\newblock {\em Applicable Analysis}, 102 (2023), no.~11, 3148--3166.

\bibitem{CostaFigueiredo2023}
G.~S. Costa and G.~M. Figueiredo,
\newblock Existence of positive solutions for $p$ \& $q$ equations involving vanishing potentials with exponential decay,
\newblock {\em Differential and Integral Equations}, 36 (2023), 859--876.


\bibitem{CostaFigueiredoJunior2023}
G.~S.~A. Costa, G.~M. Figueiredo and J.~C.~O. Junior,
\newblock Solutions for a quasilinear elliptic problem with indefinite nonlinearity with critical growth,
\newblock {\em Electronic Journal of Qualitative Theory of Differential Equations},
2023, No.~24, 1--19.


\bibitem{Seto}
S.~Seto,
\textit{The first nonzero eigenvalue of the \(p\)-Laplacian on
differential forms},
Pacific J. Math. \textbf{309} (2020), no.~1, 213--222.


\bibitem{MaranoMosconi2018}
S. A. Marano and S. J. N. Mosconi,
{\it Some Recent Results on the Dirichlet Problem for $(p,q)$-Laplace Equations}, Discrete Contin. Dyn. Syst. Ser. S {\bf 11} (2018), no. 2, 279--291.


	

\bibitem{GiustiPMC}
E.~Giusti,
\textit{Boundary value problems for non-parametric surfaces of
prescribed mean curvature},
Ann. Scuola Norm. Sup. Pisa Cl. Sci. (4) \textbf{3} (1976),
no.~3, 501--548.


\bibitem{ambrep}
V.~Ambrosio and D.~D.~Repovš,
\textit{Multiplicity and concentration results for a
\((p,q)\)-Laplacian problem in \(\mathbb R^N\)},
Z. Angew. Math. Phys. \textbf{72} (2021), no.~1,
Paper No.~33, 33~pp.

\bibitem{chil}
L.~Cherfils and Y.~Il'yasov,
\textit{On the stationary solutions of generalized reaction--diffusion
equations with \(p\&q\)-Laplacian},
Commun. Pure Appl. Anal. \textbf{4} (2005), no.~1, 9--22.

\bibitem{heli}
C.~He and G.~Li,
\textit{The regularity of weak solutions to nonlinear scalar field
elliptic equations containing \(p\&q\)-Laplacians},
Ann. Acad. Sci. Fenn. Math. \textbf{33} (2008), 337--371.

\bibitem{TavaresCarlosSalirrosas2025}
L.~Tavares, R.~D.~Carlos, and S.~M.~A.~Salirrosas,
\textit{Existence and multiplicity of solutions for \(p\&q\)
Kirchhoff--Boussinesq equations},
J. Geom. Anal. \textbf{35} (2025), Article~258.



\bibitem{Marcellini1991}
P. Marcellini,
{\it Regularity and Existence of Solutions of Elliptic Equations with $(p,q)$-Growth Conditions},
J. Differential Equations {\bf 90} (1991), 1--30.




\end{thebibliography}
\end{document}